\documentclass[11pt,a4paper,reqno]{article}
\usepackage[utf8]{inputenc}
\usepackage[shortlabels]{enumitem}
\usepackage{amsmath}
\usepackage{amssymb}
\usepackage{amsthm}
\usepackage{abstract}
\usepackage{xcolor}
\usepackage{comment}
\usepackage{mathtools}
\usepackage{graphicx}
\usepackage{caption}
\usepackage{subcaption}
\usepackage{float}
\usepackage{hyperref}
\usepackage{enumitem}
\usepackage{tikz}
\usepackage{centernot}
\usepackage{autonum}
\usepackage{framed}
\usepackage{etoolbox}
\usepackage{subcaption}

\usepackage{scalerel}

\usepackage{setspace}
\makeatletter
\def\namedlabel#1#2{\begingroup
   \def\@currentlabel{#2}%
   \label{#1}\endgroup
}
\makeatother
\usepackage{dsfont}

\definecolor{applegreen}{rgb}{0.55, 0.71, 0.0}
\definecolor{bleudefrance}{rgb}{0.19, 0.55, 0.91}
\definecolor{deepcarrotorange}{rgb}{0.91, 0.41, 0.17}
\definecolor{forestgreen}{rgb}{0.13, 0.55, 0.13}
\definecolor{linkteal}{RGB}{0,90,90}
\definecolor{linkpurple}{RGB}{60,0,90}
\definecolor{linkred}{RGB}{100,0,0}
\definecolor{linkblue}{RGB}{0,51,102}

\hypersetup{
    colorlinks=true,
    linkcolor=black,
    filecolor=black,      
    urlcolor=black,
    citecolor=black
    }
\allowdisplaybreaks
\makeatletter
\newcommand*{\barfix}[2][.175ex]{%
  \mathpalette{\@barfix{#1}}{#2}%
}
\newcommand*{\@barfix}[3]{%
  \vbox{%
    \kern#1\relax
    \hbox{$#2#3\m@th$}%
  }%
}
\makeatother

\usepackage[capitalise]{cleveref}
\newtheorem{theorem}{Theorem}
\newtheorem{thm}{Theorem}[section]
\newtheorem{corollary}[thm]{Corollary}
\newtheorem{lemma}[thm]{Lemma}
\newtheorem{proposition}[thm]{Proposition}

\newtheorem{definition}[thm]{Definition}
\theoremstyle{definition}
\newtheorem{remark}[thm]{Remark}

\newcommand{\cB}{\mathcal{B}}
\newcommand{\cC}{\mathcal{C}}
\newcommand{\cD}{\mathcal{D}}
\newcommand{\cE}{\mathcal{E}}

\newcommand{\cK}{\mathcal{K}}

\newcommand{\cM}{\mathcal{M}}

\newcommand{\cS}{\mathcal{S}}

\newcommand{\conn}{\leftrightarrow}

\newcommand{\partialin}{\partial^{-}\!}
\newcommand{\partialout}{\partial^{+}\!}

\usepackage[margin=1in]{geometry}

\def\eps{\varepsilon}

\definecolor{darkgreen}{rgb}{0.18, 0.7, 0.46}

\title{\vspace{-2em}Supercritical sharpness of percolation}
\author{%
  \small
  Sahar Diskin\footnote{Department of Mathematics, ETH Z\"urich} \;
  Philip Easo\footnote{Institute for Theoretical Studies, ETH Z\"urich, and Trinity College, University of Cambridge} \;
  Ritvik Ramanan Radhakrishnan\footnotemark[1] \;
  Benny Sudakov\footnotemark[1] \;
  Vincent Tassion\footnotemark[1]\\[0.6ex]
}

\date{}
\begin{document}
\maketitle
\vspace{-2em}\begin{abstract}
\begin{quote}
We prove that for supercritical percolation on every infinite transitive graph, the probability that the origin belongs to a finite cluster of size at least $n$ decays exponentially in $\Phi(n)$, where $\Phi$ is the isoperimetric function of the graph. 
\end{quote}
\end{abstract}
\section{Introduction}\label{s: intro}

Let $G=(V,E)$ be an infinite connected (locally-finite, vertex-)transitive graph. Let $\mathrm{P}_p$ be the standard percolation measure, under which each edge is open with probability $p$ independently of all others. Fix a vertex $o$, and write $p_c$ for the critical point given by
\begin{align}
    p_c\coloneqq\inf\{p\in [0,1]\colon \mathrm{P}_p(o\conn \infty)>0\}.
\end{align}

A fundamental result in percolation theory is subcritical sharpness (\cite{M86,AB87}, see also \cite{AV08,DT16,DCRT19, V25}), which says that for all $p<p_c$, there exists $c>0$ such that for all $n\ge 1$, the cluster $\cC_o$ of the origin satisfies
\begin{align}
    \mathrm{P}_p(|\cC_o|\ge n)\le e^{-cn} \quad \text{and} \quad \mathrm{P}_p(o\conn \partial \mathrm{B}_n)\le e^{-cn},
\end{align}
where $\mathrm{B}_n$ is the ball of radius $n$ centred at $o$, and $\partial S$ is the set of all edges with exactly one endpoint in a given set of vertices $S$.

In this paper, we will prove the following supercritical counterpart to subcritical sharpness. Write $\Phi\colon \mathbb{N}\to\mathbb{N}$ for the isoperimetric function given by
\begin{align}
    \Phi(n)\coloneqq \min\left\{|\partial S|\colon S\subset V\text{ and }n\le |S|<\infty\right\}.
\end{align}
\begin{theorem}\label{thm: exp phi}
For all $p>p_c$ there exists $c>0$ such that for all $n\ge 1$,
\begin{align}
    \mathrm{P}_p(n\le |\cC_o|<\infty)\le e^{-c\,\Phi(n)} \quad \text{and} \quad \mathrm{P}_p(o\conn\partial \mathrm{B}_n\text{ but }o\centernot\leftrightarrow\infty)\le e^{-cn}.
\end{align}
\end{theorem}

We will deduce Theorem \ref{thm: exp phi} from the following more general result. For each finite set of vertices $W$, define
\[
    \Phi(W) := \min\left\{|\partial S| \colon W \subset S \subset V\text{ and } |S|<\infty\right\}.
\]
Equivalently, $\Phi(W)$ is the size of the smallest edge cutset from $W$ to $\infty$. Also write $\cC_S$ for the set of vertices connected to $S$. 
\begin{theorem}\label{thm: main}
For all $p>p_c$ there exists $c>0$ such that for every finite set of vertices $S$ and for all $n\ge 1$,
\begin{align}
   \mathrm{P}_p(S \centernot\leftrightarrow \infty \text{ and } \Phi(\cC_S)\ge n)\le e^{-cn}.
\end{align}
\end{theorem}
\subsection{Comments}\label{s: consequences}

\paragraph{Historical context.}
In the first decades after its introduction by Broadbent and Hammersley~\cite{MR91567}, percolation was primarily studied on the hypercubic lattice $\mathbb Z^d$. Both subcritical and supercritical sharpness of the percolation phase transition were ultimately well understood in this context. For relevant textbooks, see ~\cite{Grimmetts_perco_book,BO06}. In 1996, Benjamini and Schramm launched the study of percolation on arbitrary infinite transitive graphs in their seminal paper, \emph{Percolation beyond $\mathbb Z^d$}~\cite{perco_beyond_zd}. See~\cite[Chapters 7 and 8]{Prob_on_trees_and_networks} for a survey of the rich research area that has since emerged. The proofs of certain results, such as subcritical sharpness, were easily extended to this more general setting. In our paper, we focus on supercritical sharpness, a result that is well-known in the context of the hypercubic lattice but whose proofs have remained restricted to specific geometries.


\paragraph{Previous results on supercritical sharpness.}
    In the classical case that $G= \mathbb Z^d$, supercritical sharpness was deduced in~\cite{deducing_exp_from_GM} from the celebrated Grimmett--Marstrand slab theorem~\cite{GM90} that for all $d \geq 3$,
    \[ 
            \lim_{m \to \infty} p_c \big(\mathbb Z^2 \times\{-m,\ldots,m\}^{d-2} \big) = p_c\big( \mathbb Z^{d} \big).
    \]
    At the end of our paper, we will explain why conversely, this slab theorem is easily implied by supercritical sharpness.

    More recently, supercritical sharpness was extended from $\mathbb Z^d$ to every infinite transitive graph of \emph{polynomial growth} \cite{CMT24}. At this level of generality, one cannot use renormalisation schemes that rely on the symmetries of $\mathbb Z^d$. However, the geometry of transitive graphs of polynomial growth is nevertheless known to be highly rigid following the foundational works of Gromov and Trofimov~\cite{G81,T03}. In particular,~\cite{CMT24} uses the special fact that minimal cutsets in such graphs are coarsely connected. By a different argument, supercritical sharpness was also established for every infinite transitive graph that is nonamenable \cite{HH21}, i.e.\! such that $\inf_{n\ge 1} \Phi(n)/n > 0$.

    The problem of establishing supercritical sharpness for general infinite transitive graphs has been around as folklore for some time and was recently recorded in~\cite[Conjectures 5.1--5.3]{HH21}. 

\paragraph{Tightness.} 
    Theorem \ref{thm: exp phi} is sharp in the sense that for all $p > p_c$, there exists $C < \infty$ such that for all $n\ge 1$
    \[
        \mathrm{P}_p(n\le |\cC_o|<\infty)\geq e^{-C\Phi(n)} \quad \text{and} \quad     \mathrm{P}_p(o\conn\partial \mathrm{B}_n\text{ but }o\centernot \leftrightarrow\infty)\geq e^{-Cn}.
    \]
    Indeed, the first of these inequalities was established in \cite[Comment 9]{EST25}, while the second follows by considering the event that $\cC_o$ is a fixed geodesic path of length $n$.

\paragraph{Transience.}
Following~\cite{BLS99}, there has been much work investigating which properties of the random walk on a transitive graph are stable under percolation. (See e.g.\! \cite{BM03,Peres,Gaboriau,Gabor,DisorderEntropy}, and in the context of finite transitive graphs of high degree, \cite{FR08,BKW14, EKK23,ADLZ25}). Conjecture 1.7 from \cite{BLS99} is that for every infinite transitive graph\footnote{Strictly, their conjecture was only about Cayley graphs.} $G$ that is transient, for all $p \in [0,1]$,
\begin{equation} 
    \mathrm P_p \left( \text{ every infinite cluster is transient\,} \right) =1.\label{eq:transience}
\end{equation}
Note that \eqref{eq:transience} holds vacuously when $p < p_c$, and if one believes the well-known conjecture from~\cite{perco_beyond_zd} that $\mathrm P_{p_c}(o \leftrightarrow \infty) =0$, then \eqref{eq:transience} should also hold vacuously when $p = p_c$.
By an argument from \cite[Section 3]{H22}, adapted from \cite{Gabor}, our Theorem \ref{thm: exp phi} has the following corollary.
\begin{corollary}
For every infinite connected transitive graph that is transient, \eqref{eq:transience} holds for all $p > p_c$.
\end{corollary}

 
\subsection{Notation}\label{s: notation}
We use the convention that $\mathbb{N}$ contains $0$. Let $G=(V,E)$ be a graph. Throughout the paper, we assume that every graph is locally finite.

\noindent \textbf{Graphs.} Given a set of vertices $A$, we define the following:

\begin{center}
\renewcommand{\arraystretch}{1.2}
    \begin{tabular}{| c | l|}
    \hline 
    \rule{0pt}{3ex}
    $\overline A$\; & edges with at least one endpoint in $A$ \\
    \;$A^\circ$ & edges with both endpoints in $A$ \\
    $\partial A$ &  edges with exactly one endpoint in $A$ 
\\
  $\partialin A$ & vertices in $A$ that are incident to edges in 
    $\partial A$  \\
    $\partialout A$ &  vertices in $V\setminus A$ that are incident to edges in $\partial A$ \\
    $\mathrm{B}_r(A)$ & vertices at distance at most $r$ from $A$
\\
$G[A]$ & the subgraph induced by $A$
\\

    \hline

\end{tabular}
\end{center}
\noindent \textbf{Cutsets. } Given sets of vertices $L$ and $R$, we say that a set of edges $\Pi$ is an edge cutset from $L$ to $R$ if $L$ and $R$ are not connected to each other in the graph $(V,E\setminus\Pi)$. We say that a set of vertices $A$ is a vertex cutset from $L$ to $R$ if $L\setminus A$ is not connected to $R\setminus A$ in $(V\setminus A,E\setminus\overline A)$.

\noindent \textbf{Percolation.}
For each $p\in [0,1]$, let $\mathrm{P}_p^G$, or simply $\mathrm{P}_p$, be the standard percolation measure of parameter $p$ on $G$. Given a set of vertices $A$, we may write $\mathrm{P}_p^A$ as shorthand for $\mathrm{P}_p^{G[A]}$. When working with random variables that are not defined in terms of $\mathrm{P}_p$, we think of them as all living in one big probability space called $(\Omega, \mathcal{F}, \mathbb{P})$.

\noindent \textbf{Connectivity.}  Given sets of vertices $L$ and $R$ and given a set of edges $F$, we write $L\xleftrightarrow{F}R$ for the event that $L$ is connected to $R$ in the graph $(V,F)$. As usual, we identify each configuration $\zeta \in \{0,1\}^E$ with its support, which is a set of edges. In particular, we may use expressions such as $\zeta \cap F$ and $L \xleftrightarrow{\zeta \cap F} R$. We write $L\xleftrightarrow{F}\infty$ to mean that the set of vertices connected to $L$ in the graph $(V,F)$ is infinite. We use the same definitions when $L$ or $R$ are sets of edges, in which case we naturally identify these sets with the set of vertices contained in at least one of these edges. When working with $\mathrm P_p$ specifically, given a set of vertices $A$, we write $L \xleftrightarrow{A} R$ to mean that $L \xleftrightarrow{\omega \cap A^\circ} R$, and we simply write $L \leftrightarrow R$ to mean that $L \xleftrightarrow{\omega} R$, where $\omega$ is the underlying percolation configuration.

\subsection{Outline of the paper}\label{s: outline}
In Section \ref{s: many touches}, we will introduce our key technical result, Proposition \ref{prop: main}, and we will explain why this result implies Theorem~\ref{thm: main}. In Section \ref{s: sketch}, we will present a sketch of the proof of Proposition \ref{prop: main}. In preparation for the proof of Proposition \ref{prop: main}, in Section \ref{s: finding mid-points}, we will establish a few self-contained lemmas. Section \ref{s: proof of prop} is the heart of this paper: there we will prove Proposition \ref{prop: main}. Section \ref{s: deriving theorems} is the denouement in which we will explain how to obtain Theorem \ref{thm: exp phi} from Theorem \ref{thm: main}. In the epilogue, Sections \ref{s: collecting mass} and \ref{sec:supercr-sharpn-mathb}, we will return to the classical setting of $\mathbb Z^d$ and explain how to recover Grimmett--Marstrand from supercritical sharpness.



\section{Many touches suffice}\label{s: many touches}
A key tool that is available in the study of percolation under the hypothesis that $p > p_c$, but not under the mere hypothesis that $\mathrm{P}_p(o\conn\infty )> 0$, is \emph{sprinkling}: one can often deduce a strong statement about percolation at $p$ from a weaker statement about percolation at $p' \in (p_c,p)$. In this section we will deduce Theorem \ref{thm: main} from the following a priori weaker proposition. The main goal of our paper is to prove this proposition. The proposition is formulated in terms of the following random variable. For a finite set of vertices $S$, define
\begin{align}
    \Psi(S)\coloneqq \big|\{e\in \partial S\colon e\xleftrightarrow{V\setminus S}\infty\}\big|.
\end{align}
\begin{proposition}\label{prop: main}
Let $G=(V,E)$ be an infinite connected transitive graph. Let $p\in[0,1]$ be such that $\mathrm{P}_p(o\conn\infty)>0$. For all $q\in (p,1)$ and for all $\eps\in (0,1]$, there exists $c>0$ such that for every finite set of vertices $S$ and for every positive integer $t$ with $t \leq c \, \Phi(S)$,
\begin{align}
    \mathrm{P}_q(\Psi(S)\ge t)\ge (1-\eps)^{t}.\label{eq: prop keq ineq}
\end{align}
\end{proposition}

The proof that Proposition \ref{prop: main} implies Theorem \ref{thm: main} relies on a variant of a well-known sprinkling argument: if two $p$ clusters touch at many places, then the probability that they do not merge at $p+\eps$ is small. Here we use a variant of this statement, which asserts that a $p$-cluster has low probability to touch a $(p+\eps)$-cluster at many places without being contained in it. The proof itself is a simple elaboration of an argument from \cite[Theorem 1.1, $(ii)\Rightarrow (iii)$]{H22}.

\begin{proof}[Proof of Theorem \ref{thm: main} given Proposition \ref{prop: main}]
Let $p > p_c$. 
Choose $q \in (p_c,p)$ and let $\eps>0$ be such that $(1-q)(1-\eps)=(1-p)$. By Proposition~\ref{prop: main} applied with the same parameter $q$ but with $\eps/2$ and some $p' \in (p_c,q)$ in place of ``$\eps$" and ``$p$", there exists $h>0$ such that for every finite set of vertices $C$ and for every positive integer $t$ with $t\le h\, \Phi(C)$,
\begin{equation} \label{eq:many_touches_applied}
\mathrm{P}_q(\Psi(C)\ge t)\ge \big( 1 - \eps/2\big)^{t}.
\end{equation}

Now let $S$ be a finite set of vertices and let $n \geq 1/h$. We will show that the displayed equation in Theorem \ref{thm: main} holds with $c := h\eps/4 > 0$. This is sufficient to conclude our proof because if needed, one may reduce the constant $c$ to ensure that the statement also holds for $n\le 1/h$. Define $t :=  \lfloor h n \rfloor$, and note that $t$ is a positive integer satisfying $hn/2 \leq t \leq hn$. Let $\omega$ and $\zeta$ be independent percolation configurations with $\omega \sim \mathrm{P}_q$ and $\zeta \sim \mathrm{P}_\eps$. Note that $\omega \cup \zeta \sim \mathrm{P}_{p}$ thanks to our choice of $\eps$. Consider the event 
\begin{equation}
    \mathcal{E}\coloneqq\{S \centernot{\xleftrightarrow{\omega \cup \zeta}} \infty \} \cap  \big\{\big| \{ e \in \partial \mathcal C_{S}(\omega \cup \zeta) : e \xleftrightarrow{\omega } \infty \} \big| \geq t\big\}.
\end{equation}

On the one hand, by conditioning on $\cC_S(\omega \cup \zeta)$ and applying independence,
\begin{align}
    \mathbb{P}(\cE)&=\sum_{C}\mathrm{P}_p(\cC_S=C)\cdot \mathrm{P}_q(\Psi(C)\ge t), \label{eq: lb on cE 1st}
\end{align}
where the sum is over every finite set of vertices $C$. If such a set $C$ satisfies $\Phi(C)\ge n$, then by our choice of $h$, we know that $C$ also satisfies \eqref{eq:many_touches_applied}. 
So by restricting the sum in \eqref{eq: lb on cE 1st},
\begin{align}
    \mathbb{P}(\cE)&\ge \sum_{C\colon \Phi(C)\ge n}\mathrm{P}_p(\cC_S=C)\cdot \big( 1 - \eps/2\big)^{t}\\
    &=\mathrm{P}_p(S \centernot\leftrightarrow \infty \text{ and }\Phi(\cC_S)\ge n)\cdot \big( 1 - \eps/2\big)^{t}.\label{eq: lb on cE}
\end{align}

On the other hand,
\begin{align} 
    \mathbb{P}(\cE)\le \left(1-\eps\right)^{t}. \label{eq: ub on cE}
\end{align}
Indeed, this is essentially \cite[Proposition 2.1]{H22}, but let us anyway sketch a proof for completeness. Define
\[
    \mathcal K := \left\{ v \in V: v\xleftrightarrow{\omega}\infty\right\} \qquad \text{and} \qquad \mathcal M\coloneqq \cC_{S\backslash\cK}\left((\omega \cup \zeta) \setminus \overline \cK\right).
\]
Observe that
\[
    \cE=\{\cK\cap S=\varnothing\}\cap \{|\partial \cK\cap \partial \cM|\ge t\}\cap \{\zeta \cap \partial \cK\cap \partial \cM=\varnothing\}.
\]
Now a.s. if $\cK\cap S=\varnothing$ and $|\partial \cK\cap \partial \cM|\ge t$, then by independence,
\begin{align}
    \mathbb{P}(\cE\mid \cK, \cM)&=\mathbb{P}(\zeta \cap \partial\cK\cap \partial\cM=\varnothing\mid \cK, \cM)=\left(1-\eps\right)^{|\partial\cK\cap \partial\cM|}\le \left(1-\eps\right)^{t},\label{eq: sketch}
\end{align}
which immediately yields \eqref{eq: ub on cE} upon taking expectations.


Now by \eqref{eq: lb on cE}, \eqref{eq: ub on cE}, and the fact that $t \geq hn/2$, we obtain
\begin{align}
    \mathrm{P}_p(S\centernot\leftrightarrow\infty\text{ but }\Phi(\cC_S)\ge n)\le \left(\frac{1-\eps}{1-\eps/2}\right)^{t} \leq e^{- \frac{\eps t}{2}} \leq e^{-\frac{ \eps h n }{4}}= e^{-cn}.&\qedhere
\end{align}
\end{proof}

\section{Sketch of proof}\label{s: sketch}
In this section, let us present a sketch of the proof of Proposition \ref{prop: main}. Fix a connected finite set $S$ in the graph $G$ and consider percolation restricted to the complement of $S$. 
Call an edge $e\in \partial S$ a `touch' if it is connected to infinity (in this restricted configuration). Our proof relies on an exploration algorithm that witnesses these touches one after another. An important feature of our algorithm is that the set of revealed edges of our $p$-percolation remains connected throughout the process. We will first describe a toy version of this algorithm and then explain how to modify it for the final construction.

\paragraph{Toy algorithm.} Pick a vertex that is midway between $S$ and infinity, and explore its cluster $\cC_1$. Here \emph{midway} means that the vertex has constant probability to connect to both $S$ and $\infty$. Then pick a vertex that is midway between $S$ and $\mathcal{C}_1$, and explore its cluster $\cC_2$. Note that at this stage, $\cC_2$ can at best reach the external vertex boundary of  $\cC_1$ because the boundary of $\cC_1$ has been revealed to be closed. Consider all edges where $\cC_1$ and $\cC_2$ touch, and sprinkle here to glue these two clusters together with constant probability. Continue in this manner $t$ times, obtaining sets $\cC_1,\dots,\cC_t$, at each step choosing a point that is midway between $S$ and the previously explored sets. This construction guarantees that each exploration has at least constant probability $\alpha>0$ to connect to the previous explorations and witness a new touch. It follows that \[\mathrm{P}_q(\Psi(S)\geq t)\geq \alpha^t.\]

\paragraph{Improved algorithm.}  To improve the probability of creating a new touch from $\alpha$ to $1-\varepsilon$, we split each iteration into two: First, by revealing one large ball at a time, we search for a \emph{seed} (a fully open ball) that is midway between $S$ and the previous explorations. Then we explore the cluster of this seed. To ensure that the search for seeds does not interact with future cluster explorations, we look for seeds in a \emph{seed layer}, which is simply a low-density percolation that is independent of the original $p$-percolation.

\paragraph{Where does $\Phi$ come into play?}
Let us conclude this sketch with an important observation. In the toy algorithm, the set of all revealed edges is always connected. In contrast, in the improved algorithm, the search for seeds generates many \emph{bad} balls in the seed layer that are disconnected from each other yet contain negative information. Consequently, while the toy algorithm can always be run until it witnesses all possible touches (i.e.\! all edges in the exposed boundary of $S$), the improved algorithm may halt after finding only order $\Phi(S)$ touches if the union of the bad balls forms a cutset from $S$ to infinity. 
\vspace{0.5cm}
\begin{figure}[H]
\centering
\begin{subfigure}{0.48\textwidth}
    \centering
    \includegraphics[
        width=0.7\linewidth,
        trim={6cm 6cm 12cm 4cm}, 
        clip
    ]{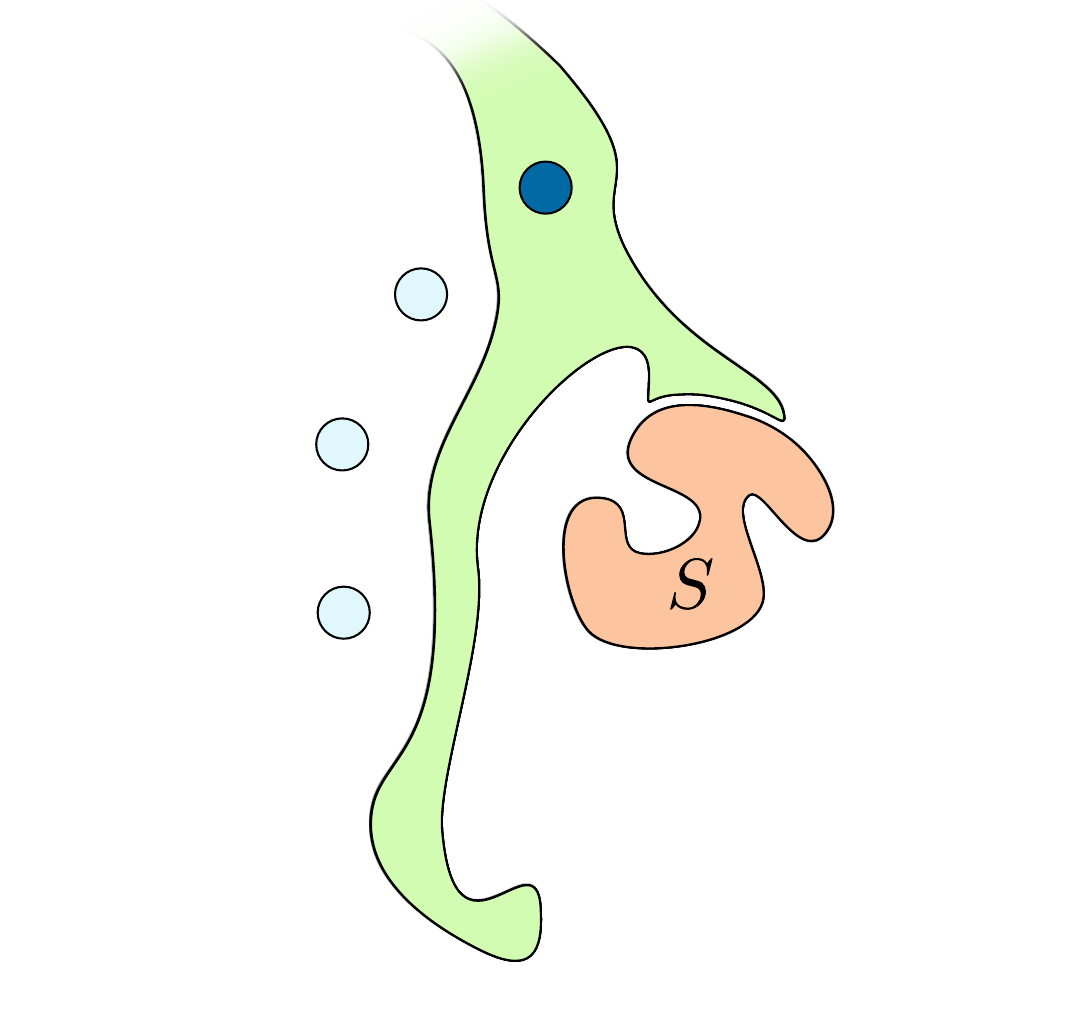}
\end{subfigure}
\hfill
\begin{subfigure}{0.48\textwidth}
    \centering
    \includegraphics[
        width=0.7\linewidth,
        trim={12cm 6cm 6cm 4cm}, 
        clip
    ]{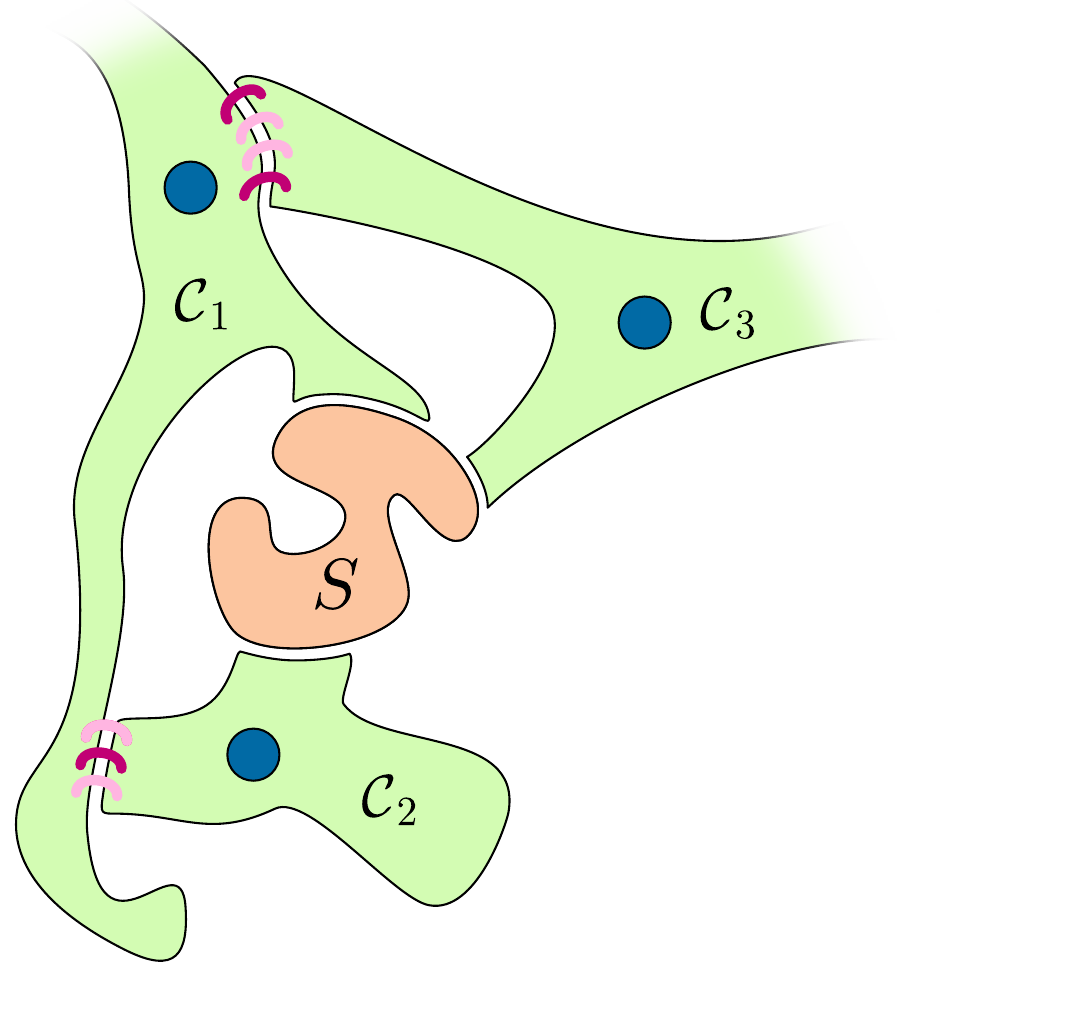}
\end{subfigure}

\caption{On the left, the blue balls are where the seed layer was revealed: each light blue ball was found to contain at least one closed edge, whereas the dark blue ball was found to be entirely open, forming a \emph{seed}. The green region is the $p$-cluster of the seed, which connects to infinity and touches the peach-coloured set $S$. On the right, we have repeated the process of finding a seed and exploring its cluster twice more. The pink edges represent the sprinkling used to glue each new green cluster to the existing green clusters: the light pink edges are closed, and the dark pink edges are open.}
\end{figure}

\section{Finding mid-points and mid-balls}\label{s: finding mid-points}

A standard observation in percolation theory is that for any sets of vertices $L$ and $R$ in a connected graph $G$, if every vertex is connected to $L \cup R$ with constant probability, then there exists a \textit{mid-point} vertex that has both constant probability to connect to $L$ and to $R$. (See \cite{EST25} for a recent example.) 

Now suppose that our graph contains not just one but many disjoint paths from $L$ to $R$. Then we can find not just one but many mid-points. In particular, even if we are given a not-too-large set of vertices $A$ that we seek to avoid, there will be at least one mid-point that is disjoint from $A$. In this section, we present a high probability version of the above argument where we construct \textit{mid-balls} instead of mid-points. We begin with the construction of one mid-ball that avoids a forbidden region $A$.

\begin{lemma} \label{lem:finding_one_seed}
    Let $G=(V,E)$ be a graph, and let $L$ and $R$ be disjoint sets of vertices. Let $r\ge 1$ and let $A$ be a set of vertices such that $\mathrm{B}_{r}(A)$ is not a vertex cutset from $L$ to $R$. Let $p, \delta \in [0,1]$ be such that for all $v \in V$,
    \begin{equation}\label{eq:deltahypothesis}
        \mathrm{P}_p(\mathrm{B}_{r-1}(v) \leftrightarrow L \cup R)\geq 1-\delta.
    \end{equation}
    Then there exists a ball $B$ of radius $r$ that is disjoint from $A$ and satisfies
    \begin{equation} \label{eq:halfway}
        \mathrm{P}_p(B \xleftrightarrow{V \backslash R} L), \ \mathrm{P}_p(B \xleftrightarrow{V \backslash L} R) \geq 1 - \sqrt{\delta}.
    \end{equation}
\end{lemma}
%
\begin{proof}
Since $\mathrm{B}_{r}(A)$ does not form a vertex cutset from $L$ to $R$, there exists a path $v_1,\dots,v_k$ from $L$ to $R$ that does not visit $\mathrm{B}_{r}(A)$. Note that $\mathrm{B}_r(v_i)$ is disjoint from $A$ for all $i$. By the FKG inequality ($\sqrt{\mathsf{trick}}$) and~\eqref{eq:deltahypothesis}, for all $i$,
\[
    \mathrm{P}_p(\mathrm{B}_{r-1}(v_i) \xleftrightarrow{V \backslash R} L) \geq 1- \sqrt{\delta} \quad \text{or} \quad \mathrm{P}_p(\mathrm{B}_{r-1}(v_i) \xleftrightarrow{V \backslash L} R) \geq 1 - \sqrt{\delta}.
\]
Note that $v_1$ satisfies the first option and $v_k$ satisfies the second. So we can find some $i \in \{1,\dots,k-1\}$ such that
\[
    \mathrm{P}_p(\mathrm{B}_{r-1}(v_i) \xleftrightarrow{V \backslash R} L),\ \mathrm{P}_p(\mathrm{B}_{r-1}(v_{i+1}) \xleftrightarrow{V \backslash L} R) \geq 1 - \sqrt{\delta}.
\]
Since $\mathrm{B}_{r}(v_i)$ contains $\mathrm{B}_{r-1}(v_i)$ and $\mathrm{B}_{r-1}(v_{i+1})$, the ball $B := \mathrm{B}_r(v_i)$ satisfies \eqref{eq:halfway}.
\end{proof}

By adding a stronger cutset condition and repeatedly applying the above lemma, we next construct \emph{many} disjoint mid-balls between two sets.

\begin{lemma} \label{lem:finding_many_seeds}
    Let $G=(V,E)$ be a graph, and let $L$ and $R$ be disjoint sets of vertices. Let $r,\ell\ge 1$ and let $A$ be a finite set of vertices such that every vertex cutset from $L$ to $R$ has size at least
    \begin{equation} \label{eq:min_cutset_size}
        \sup_{v \in V} | \mathrm{B}_{2r}(v)| \cdot \big( |A| + \ell \big).
    \end{equation}
    Let $p,\delta\in [0,1]$ be such that \eqref{eq:deltahypothesis} is satisfied. Then there exist balls $B_1,\dots,B_{\ell}$ of radius $r$ that are disjoint from each other and from $A$ such that for all $i$
        \begin{equation} 
        \mathrm{P}_p(B_i \xleftrightarrow{V \backslash R} L), \ \mathrm{P}_p(B_i \xleftrightarrow{V \backslash L} R) \geq 1 - \sqrt{\delta}.\label{eq: ritvik}
    \end{equation} 
\end{lemma}
\begin{proof}
    By induction, it suffices to show that for all $i \in \{1,\dots,\ell\}$, given any balls $B_1,\dots,B_{i-1}$ of radius $r$, there exists a ball $B_i$ of radius $r$ that satisfies \eqref{eq:halfway} and is disjoint from
    \[
        A_i := A \cup B_1 \cup \cdots \cup B_{i-1}.
    \]
    Lemma \ref{lem:finding_one_seed} guarantees that such a ball $B_i$ exists if we are able to verify that $\mathrm{B}_{r}(A_i)$ is not a vertex cutset from $L$ to $R$. Let $u_j$ be the centre of the ball $B_j$ for each $j \leq i-1$, and let $U := A \cup \{u_1,\dots, u_{i-1}\}$. Since $\mathrm{B}_r(A_i) \subset \mathrm{B}_{2r}(U)$ and $|U| < |A| + \ell$,
    \[
    |\mathrm{B}_r(A_i)|\le \max_{v \in U} | \mathrm{B}_{2r}(v)|\cdot |U|<\sup_{v \in V} |\mathrm{B}_{2r}(v)|\left(|A|+\ell\right).
    \]
   So by \eqref{eq:min_cutset_size}, $\mathrm{B}_r(A_i)$ is indeed not a vertex cutset from $L$ to $R$.
\end{proof}

When we apply Lemma \ref{lem:finding_many_seeds} in the proof of Proposition \ref{prop: main}, edges on the boundary of the set called $R$ will have already been revealed to be closed. Nevertheless, we will want to connect a mid-ball given by Lemma \ref{lem:finding_many_seeds} to $R$ with high probability. To salvage the situation, we will sprinkle by a small amount $\eta > 0$ on the boundary of $R$ and apply the following standard observation about the (non-)effect of thinning on connection probabilities.

\begin{lemma}\label{lem: sprinkling} 
For all $p,\eta \in (0,1)$ and $\varepsilon >0$, there exists $\alpha>0$ such that the following holds. Let $B$ and $R$ be sets of vertices in a graph $H=(V,E)$. Consider independent configurations
\begin{equation}
\omega\sim \mathrm{P}_p^{V\setminus R} \quad \text{and} \quad \xi \sim \mathrm{P}_\eta^{\partial R}. 
\end{equation}
If $\mathrm{P}_p(B\leftrightarrow R)\geq 1-\alpha$, then
\[
    \mathbb{P}(B\xleftrightarrow[]{\,\omega \cup \xi\,} R)\geq 1-\varepsilon.
\]
\end{lemma}
\begin{proof}
The conclusion is trivial if $B$ intersects $R$, so let us assume that $B$ does not intersect $R$.
Let $T$ be the (random) number of edges $e\in \partial R$ such that $e\xleftrightarrow{\omega}B$. Note that
\begin{align}
    \mathrm{P}_p(B\centernot \leftrightarrow R)=\mathbb{E}\left[(1-p)^T\right]\quad\text{and}\quad\mathbb{P}(B\centernot{\xleftrightarrow{\omega\cup \xi}}R)=\mathbb{E}\left[(1-\eta)^T\right].
\end{align}
In particular, for any constant $t\ge 0$,
\begin{align}
    \mathrm{P}_p(B\centernot \leftrightarrow R)\ge (1-p)^t \cdot \mathbb{P}(T\le t),
\end{align}
and hence for any $\alpha>0$ satisfying $\mathrm{P}_p(B\centernot \leftrightarrow R)\le \alpha$,
\begin{align}
    \mathbb{P}\left(B\centernot{\xleftrightarrow{\omega\cup \xi}}R\right)&=\mathbb{E}\left[(1-\eta)^T\mathrm{1}_{T\le t}\right]+\mathbb{E}\left[(1-\eta)^T\mathrm{1}_{T>t}\right]\\
    &\le \mathbb{P}(T\le t)+(1-\eta)^t\le \frac{\alpha}{(1-p)^t}+(1-\eta)^t.
\end{align}
It now suffices to first pick $t$ large enough to ensure that $(1-\eta)^t\le \frac{\eps}{2}$, then pick $\alpha>0$ small enough to ensure that $\frac{\alpha}{(1-p)^t}\le\frac{\eps}{2}$.
\end{proof}

\section{Proof of Proposition \ref{prop: main}}\label{s: proof of prop}

Let $q \in (p,1)$ and $\varepsilon\in (0,1]$. 
Our goal is to find a constant $c>0$ such that for every finite set of vertices $S$ and for every positive integer $t$ with $t\le c\, \Phi(S)$,
\begin{align}
 \mathrm{P}_q(\big|\{e\in \partial S\colon e\xleftrightarrow{V\setminus S}\infty\}\big|\ge t)=\mathrm{P}_q\big(\sum_{e\in \partial S}\mathrm{1}\{e\xleftrightarrow{V\setminus S}\infty\}\ge t\big)\ge (1-\eps)^{t}. \label{eq:propkeqineq}
\end{align}

The proof is organised as follows: In Section~\ref{sec:41}, we will define various constants and ultimately the desired constant $c>0$.  These constants depend only on $G$, $p$, $q$, and $\varepsilon$.  In Section~{\ref{sec:42}}, we will reduce our goal to a finite volume version. The core content of the proof is then presented in Sections \ref{sec:43}--\ref{s: creating}.

\subsection{Definition of the constant \texorpdfstring{$c$}{c}}\label{sec:41}
Pick $\eta\in (0,1)$ such that
\begin{align}\label{eq: poor equation}
    (1-p)(1-\eta)^2= 1-q,
\end{align}
so that the union of $p$-percolation and two copies of $\eta$-percolation has the law of $q$-percolation. Let \begin{equation}\label{eq: delta}
    \delta\coloneqq \min(\alpha,\eps/4)^2,
\end{equation} where $\alpha$ is the constant guaranteed by Lemma \ref{lem: sprinkling} applied with $p$, $\eta$, and $\eps/4$ in place of ``$\eps$''. Pick $r\ge 1$ such that 
\begin{align}\label{eq: r}
    \mathrm{P}_p(\mathrm{B}_{r-1}\conn \infty)\ge 1-\delta,
\end{align}
which exists because $\mathrm{P}_p(o\conn \infty)>0$ and hence $\lim_{R\to\infty}\mathrm{P}_p(\mathrm{B}_R\conn \infty)=1$. Define
\[
    b := \max\left(| \mathrm{B}_r^\circ|, | \mathrm{B}_r|\right),
\]
and choose $\ell\ge 1$ such that
\begin{align}\label{eq: eta}
    \big(1-\eta^{b}\big)^\ell\le \frac{\eps}{2}.
\end{align}
In particular, this ensures that if we do $\eta$-percolation on $\ell$-many disjoint balls of radius $r$, then with probability at least $1-\eps/2$, we can find a ball that is entirely open.
Let $d$ be the vertex degree of $G$, and note that $d \geq 2$.
We will show that we can take 
\begin{align}\label{eq: choice of c}
    c\coloneqq \frac{1}{4b^3 d \ell}.
\end{align}

\subsection{Reduction to finite volume} \label{sec:42}
 Let $S$ be a finite set of vertices, and let $t$ be a positive integer such that $t \leq c\, \Phi(S)$ (which we may assume exists; otherwise the conclusion holds vacuously). Let us start by reducing \eqref{eq:propkeqineq} to a statement about percolation in a large ball.
A.s., for each edge $e\in \partial S$, the indicator of the event $\big\{e\xleftrightarrow{V\setminus S}\partial \mathrm{B}_{R+1}\big\}$ is decreasing in $R$ when $\mathrm{B}_R \supset S$, and converges to the indicator of the event $\big\{e\xleftrightarrow{V\setminus S}\infty\big\}$ as $R\to\infty$. So it suffices to show that for all $R$ such that $\mathrm{B}_R\supset S$,
\begin{align}
    \mathrm{P}_q\big(\sum_{e\in \partial S}\mathrm{1}\{e\xleftrightarrow{V\setminus S}\partial\mathrm{B}_{R+1}\}\ge t\big)\ge (1-\eps)^{t}.\label{eq: reducing to ball}
\end{align}

For the rest of the proof, fix $R$ such that $\mathrm{B}_{R}\supset S$, and set $\Lambda\coloneqq \mathrm{B}_{R+1}$. Let $\omega$, $\xi,$ and $\zeta$ be independent (random) configurations such that 
\[
\omega \sim \mathrm{P}_p^{\Lambda\setminus S}\quad \text{ and }\quad \xi,\zeta\sim\mathrm{P}_{\eta}^{\Lambda\setminus S}.
\]
Note that by \eqref{eq: poor equation}, $\omega\cup \xi\cup \zeta\sim \mathrm{P}_q^{\Lambda\setminus S}$. So it suffices to show that
\begin{align}
    \mathbb{P}\big( |\partial \cC_{\partial \Lambda} \cap \partial S| \ge t\big)\ge (1-\eps)^{t},\label{eq: key estimate}
\end{align}
where $\cC_{\partial \Lambda}$ is the set of all vertices in $\Lambda \backslash S$ that are connected to $\partial \Lambda$ in $\omega \cup \xi \cup \zeta$.

\subsection{It suffices to create touches one by one}\label{sec:43}

In this section, we deduce \eqref{eq: key estimate} from the following lemma, which isolates the relevant properties of a sequence of clusters that we will construct by the successive explorations described in Section \ref{s: sketch}. We will prove the lemma itself in the subsequent subsections.
\begin{lemma}\label{lem: induction}
There exists a sequence of (random) sets of vertices $\cC_0,\dots,\cC_t \subset  \cC_{\partial \Lambda}$ such that the following holds. For each $i\in \{0,\ldots, t\}$, define
\begin{equation} \label{eq:def_of_Ni}
    N_i := |\partial (\cC_0 \cup \dots \cup \cC_i) \cap \partial S|.
\end{equation}
Then for all $i\in \{1,\ldots, t\}$, we have
\begin{equation}\label{eq:touches_++}
    \mathbb{P} ( N_{i-1} < N_i  \mid N_0 < \cdots < N_{i-1} < t) \geq 1-\eps,
\end{equation}
or the event being conditioned on has probability zero.
\end{lemma}

\begin{proof}[Completing the proof of Proposition \ref{prop: main} given Lemma~\ref{lem: induction}] Let $\cC_0,\dots,\cC_t$ be the sequence given by Lemma \ref{lem: induction}, and for all $i$ let $N_i$ be as in \eqref{eq:def_of_Ni}. For each $i\in \{0,\ldots, t\}$, consider the event
\[
   \mathcal E_i:=\{  N_0 < \cdots < N_i \} \cup\{ N_{i-1} \geq t\},
\]
where we set $N_{-1} := 0$. Let $i \geq 1$.
By \eqref{eq:touches_++},
\[
    \mathbb P(N_0 < \dots < N_i \text { and } N_{i-1} < t) \geq (1-\eps) \cdot \mathbb P(N_0 < \dots < N_{i-1}< t). \label{eq: ni desc}
\]
Since $N_{i-2}\leq N_{i-1}$, we have the inclusion
\begin{equation}
     \mathcal{E}_{i-1}\subset \{N_0<\cdots<N_{i-1}<t\}\cup \{N_{i-1}\geq t\}. \label{eq: inclusion stuff}
\end{equation}
Therefore, by applying the definition of $\mathcal E_i$ and by \eqref{eq: ni desc} and \eqref{eq: inclusion stuff}, 
\begin{align}
    \mathbb P(\mathcal E_i) &= \mathbb P(N_0 < \dots < N_i \text { and } N_{i-1} < t) + \mathbb P(N_{i-1} \geq t) \\
    &\geq (1-\eps) \cdot \big( \mathbb P(N_0 < \dots < N_{i-1}< t) + \mathbb P(N_{i-1} \geq t)\big)\\
    &\geq (1-\varepsilon)\cdot \mathbb{P}(\mathcal{E}_{i-1}).
\end{align} 
In particular, noting that $\mathcal E_t \subset \{N_t \geq t\}$ and $\mathbb{P}(\mathcal{E}_0)=1$,
\begin{align}
    \mathbb P( |\partial C_{\partial \Lambda} \cap \partial S| \ge t) \geq \mathbb P(\mathcal E_t) \geq (1-\eps)\cdot \mathbb P(\mathcal E_{t-1}) \geq \cdots \geq (1-\eps)^t. 
\end{align}
This establishes \eqref{eq: key estimate} and thus concludes the proof of the proposition.
\end{proof}

\subsection{Choosing potential seeds}\label{sec:44}
In this subsection, we apply Lemma~\ref{lem:finding_many_seeds} to construct a function $\mathsf{Seeds}$ that we will use to produce a collection of mid-balls between what has already been explored and the fixed set $S$. We emphasise that the function $\mathsf{Seeds}$ appearing in this lemma is purely deterministic. Let us fix a choice of such a function for all subsequent subsections. Given a set of vertices $A$, let $\mathcal{P}(A)$ denote the set of all subsets of $A$. 

 
\begin{lemma}\label{lem:seeds}
There exists a function
\begin{equation}
     \mathsf{Seeds} : \mathcal P(\Lambda \backslash S)^2 \to \mathcal{P}(\Lambda \backslash S)^{\ell}
\end{equation}
with the following property. For every input $(D,K)$ satisfying
\begin{equation}\label{eq:sat}
K \supset \partialin \Lambda, \quad |\partial K \cap \partial S| < t, \quad \text{and} \quad |D| \leq b\ell t,
\end{equation}
the output is a sequence of sets $B_1,\dots,B_\ell$ that are disjoint from each other and from $D$ such that for all $i$, the set $B_i$ is a ball of radius $r$ with respect to the graph $G[\Lambda\backslash S]$ and satisfies
\begin{equation} \label{eq:halfway_in_our_context}
    \mathrm{P}_p\bigl(B_i \xleftrightarrow[]{\Lambda\setminus (K \cup S)} \partialout S\setminus K\bigr),\, \mathrm{P}_p\bigl(B_i \xleftrightarrow[]{\Lambda\setminus S} K\bigr)
\ge 1-\sqrt{\delta}.
\end{equation}
\end{lemma}

\begin{proof}
Suppose that we are given sets $D,K \subset \Lambda \backslash S$ satisfying \eqref{eq:sat}. We would like to define $\mathsf{Seeds}(D,K)$ to be the sequence of balls given by applying Lemma~\ref{lem:finding_many_seeds} to the graph \(G[\Lambda \backslash S]\), using the same $p$, $\delta$, $r$, $\ell$, but with $D$ in place of ``$A$'', and using
\[
    L := \partialout S\setminus K \quad \text{and} \quad R:= K,
\]
since~\eqref{eq:halfway_in_our_context} is immediately implied by~\eqref{eq: ritvik} with this setup. (There may be multiple sequences satisfying the conclusion of Lemma~\ref{lem:finding_many_seeds}; to define $\mathsf{Seeds}(D,K)$, we choose a particular sequence from among all options according to an arbitrary deterministic rule.)
To apply this lemma, we need to \textbf{a)} verify \eqref{eq:deltahypothesis} and \textbf{b)} show that every vertex cutset in $G[\Lambda \backslash S]$ from $L$ to $R$ has size at least~\eqref{eq:min_cutset_size}.

\noindent \textbf{a)} Let $v\in \Lambda\setminus S$. Consider $B\coloneqq \mathrm B_{r-1}^{G[\Lambda \setminus S]}(v)$ and $B'\coloneqq \mathrm B_{r-1}^{G}(v)$, and note that either $B = B'$ or $B'$ is not entirely contained in  $\Lambda \setminus S$. If $B = B'$, then by our choice of $r$ in~\eqref{eq: r},
\begin{align}
    \mathrm{P}_p(B\xleftrightarrow{\Lambda \setminus S} L \cup R) &\geq \mathrm{P}_p(B\xleftrightarrow{\Lambda \setminus S} 
    \partialout S \cup \partialin \Lambda) \\&\geq \mathrm{P}_p(B \leftrightarrow \infty) = \mathrm P_p(B' \leftrightarrow \infty) \geq 1-\delta.
\end{align}
If $B'$ is not entirely contained in $\Lambda \setminus S$, then $B$ must intersect $\partialout S \cup \partialin \Lambda \subset L\cup R$, and hence $\mathrm{P}_p(B\xleftrightarrow{\Lambda \setminus S} L \cup R) = 1.$

\noindent \textbf{b)} Observe that every ball of radius \(2r\) in \(G[\Lambda \setminus S]\) is contained in a ball of radius \(2r\) in $G$ and hence has size at most $|\mathrm B_{2r}| \leq b^2$.
Conversely, any vertex cutset $W$ in $G[\Lambda \setminus S]$ from $L$ to $R$ induces an edge cutset $\overline W\cup (\partial K\cap \partial S)$ in $G$ that from \(S\) to infinity and hence satisfies
\[
    d\cdot |W|+|\partial K\cap \partial S| \geq |\overline W|+|\partial K\cap \partial S| \geq \Phi(S).
\]
So it suffices to show that
\begin{equation}\label{eq:ba_vs_S}
d\cdot b^2(|D|+\ell)+|\partial K\cap \partial S| < \Phi(S).
\end{equation}
By~\eqref{eq:sat}, since $t\le c\, \Phi(S)$ and by our choice of $c$ in \eqref{eq: choice of c},
\begin{align} \label{eq:bound_on_A}
    d\cdot b^2|D|+|\partial K \cap \partial S | < db^3\ell t+t \le 2d b^3\ell t \leq  2db^3\ell \cdot c\, \Phi(S) = \frac{1}{2} \Phi(S).
\end{align}
Furthermore, since $t$ is a positive integer satisfying $t \leq c \, \Phi(S)$, we have that $\Phi(S)\ge 1/c$, and hence by our choice of $c$ in \eqref{eq: choice of c}
\begin{equation}\label{eq:bl_vs_phi}
d\cdot b^2\ell \leq \frac{1}{2} \cdot \frac{1}{c} \le \frac{1}{2}  \Phi(S).
\end{equation}
By combining~\eqref{eq:bound_on_A} with~\eqref{eq:bl_vs_phi}, we obtain \eqref{eq:ba_vs_S} as required.
\end{proof}
\begin{figure}[H]
\centering
\includegraphics[width=0.5\textwidth, 
trim={4cm 4cm 4cm 4cm}
]{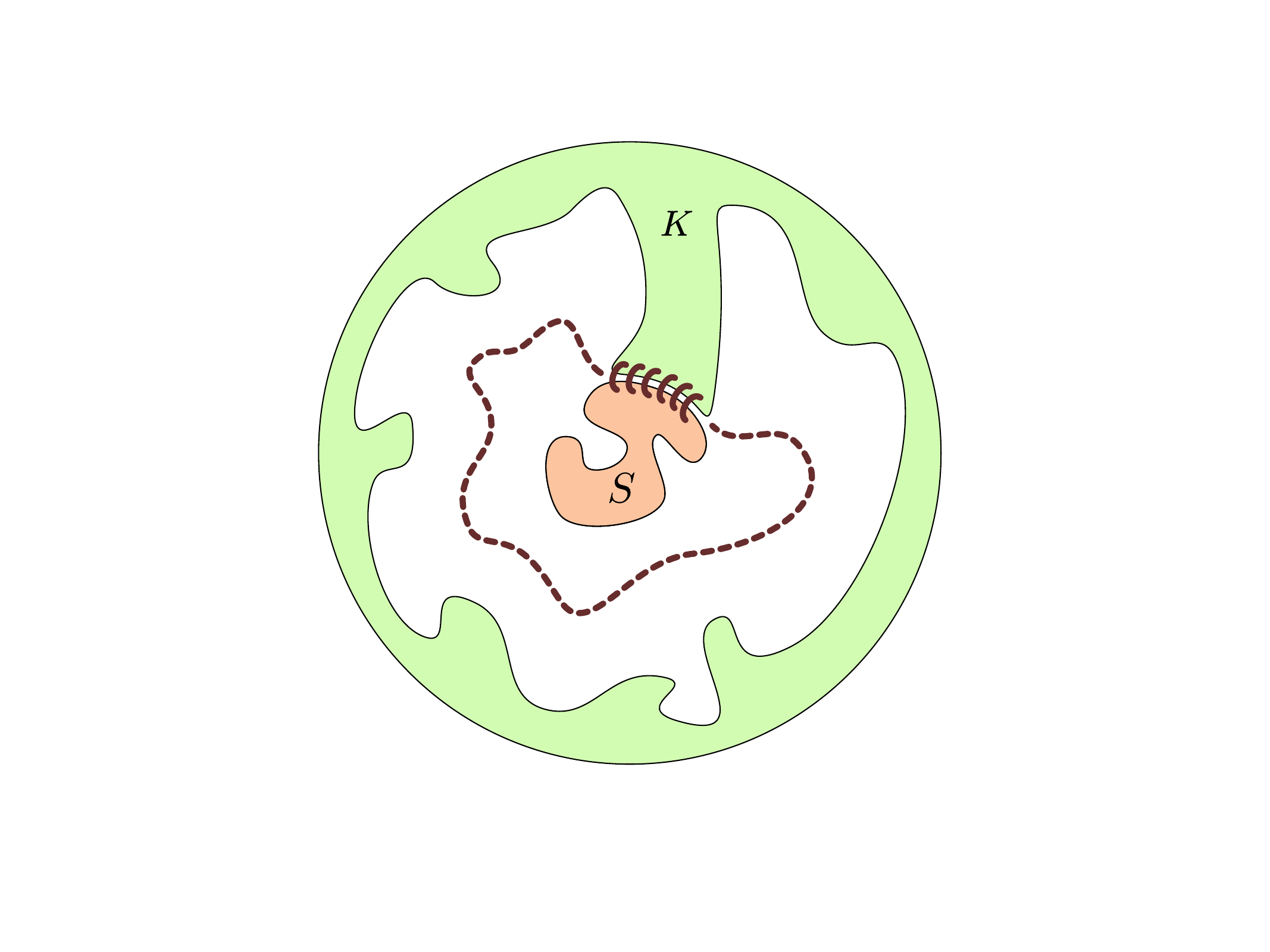}
\caption{
The dashed brown line is \textit{midway} between the peach set $S$ and the green set $K$ in the sense that any $r$-ball centred on this line has high probability to connect to both $\partialout S\setminus K$ and to $K$. This line cannot be too short because together with the brown edges making up $\partial S\cap \partial K$, we have a cutset from $S$ to $K$ and hence to infinity. Therefore, even if we are given a not-too-big region $D$ (not displayed), we can find many $r$-balls centred along the dashed line that are disjoint from each other and from $D$. We define $\mathsf {Seeds}(D,K)$ to be the sequence of such balls.}

\label{f: midway}
\end{figure}


Let us now estimate the probability that an exploration of the cluster of a given seed $B$ connects to both $K$ and $\partialout S\setminus K$. However, in the proof of Lemma \ref{lem: induction}, we will have already revealed the status of many edges in $\overline K$. Therefore, we will not allow connections to take place in the whole of $\omega \cup \zeta \cup \xi$, but rather in a \emph{fresh} subset of these configurations that will not yet have been revealed.

\begin{lemma} \label{lem:Bconnectstoblahwithgoodproba}
Let $D,K\subset \Lambda\setminus S$ be sets satisfying~\eqref{eq:sat}, and let $B\in \mathsf{Seeds}(D,K)$. Then
\begin{equation}\label{eq:caseb.2}
    \mathbb{P}(\partialout S\setminus K \xleftrightarrow[]{\omega\setminus \overline{K}} B \xleftrightarrow[]{(\omega\setminus\overline{K})\cup (\xi\cap \partial K)} K)\geq 1-\frac{\varepsilon}{2}.
\end{equation}
\end{lemma}
\begin{proof}
We first show that \begin{equation}\label{eq:BconnectsK}
    \mathbb{P}(B \xleftrightarrow[]{(\omega\setminus\overline{K})\cup (\xi\cap \partial K)} K)\geq 1-\frac{\varepsilon}{4}.
\end{equation}
By the second part of~\eqref{eq:halfway_in_our_context} and by the definition of $\delta$ in~\eqref{eq: delta},
\begin{equation} \label{eq:b_t_k_much_alpha}
\mathrm{P}_p(B\xleftrightarrow[]{\Lambda\setminus S} K)\geq 1-\sqrt\delta\ge 1-\alpha.
\end{equation}
Recall that $\alpha$ was chosen according to the sprinkling lemma (Lemma~\ref{lem: sprinkling}) with parameters $p$, $\eta$, and $\eps/4$ in place of `$\eps$'. By applying the conclusion of the sprinkling lemma to the graph $G[\Lambda\backslash S]$ with $B$ and $K$ in place of ``$B$'' and ``$R$'', we obtain \eqref{eq:BconnectsK}. Moreover, by the first part of~\eqref{eq:halfway_in_our_context} and by the definition of $\delta$ in~\eqref{eq: delta},
\begin{equation}\label{eq:BconnectsS}
    \mathbb{P}(B\xleftrightarrow{\omega \setminus \overline K}\partialout S\setminus K)\ge 1-\sqrt{\delta}\ge 1-\frac{\eps}{4}.
\end{equation}
The claimed inequality now follows from~\eqref{eq:BconnectsK} and~\eqref{eq:BconnectsS} by a union bound.
\end{proof}

\subsection{Growing seeds}\label{s: growing}
In this section, given a ball outputted by the $\mathsf{Seeds}$ function from the previous subsection, we will describe how to construct a certain cluster of the ball via a two-step exploration.

\begin{definition}
Given sets $D,K \subset \Lambda \setminus S$ satisfying \eqref{eq:sat}, and given $B\in \mathsf{Seeds}(D,K)$, define
$$\cC_{D,K}(B)\coloneqq \cC_1\cup \cC_2,$$
where
\begin{equation}
     \mathcal C_1:= \cC_{B}(\omega \setminus \overline K) \qquad \text{and} \qquad \mathcal C_2 := V\left(\partial\left(\cC_1\setminus K\right) \cap \partial K \cap \xi\right),
\end{equation}
i.e.\! $\cC_2$ is the set of all vertices contained in a $\xi$-open edge $uv$ with $u \in \cC_1\backslash K$ and $v \in K$.
\end{definition}
Here are three elementary properties of $\mathcal C_{D,K}(B)$ that are evident from its construction.

\begin{remark}\label{rem:Cconnected} Almost surely, $\mathcal{C}_{D,K}(B)$ is connected in $\omega \cup \xi \cup B^\circ.$ 
\end{remark}
\begin{remark}\label{lem:equivalent1}
Almost surely, $\mathcal{C}_{D,K}(B) \text{ intersects } \partialout S \setminus K$ if and only if
\[
    B\xleftrightarrow[]{\omega\backslash \overline{K}} \partialout S\setminus K,
\]
and $\mathcal{C}_{D,K}(B) \text{ intersects $K$}$ if and only if
\[
    B \xleftrightarrow[]{(\omega\setminus\overline{K})\cup (\xi\cap \partial K)} K.
\]
\end{remark}

 \begin{remark}\label{rem:Cmeas} For every fixed set $C\subset \Lambda \setminus S$, the event $\{\mathcal{C}_{D,K}(B)=C\}$ is measurable with respect to
 \[
    \omega \cap \overline C \qquad \text{and} \qquad \xi \cap \partial(C\setminus K)\cap \partial K.
 \]
 \end{remark}

\subsection{Proof of Lemma~\ref{lem: induction}}\label{s: creating}
Building upon the previous subsections, we can now construct our sequence of random sets.

\noindent\textbf{Construction of $\cC_0,\dots,\cC_t$}
\\Let us recursively define a sequence of sets $\cC_0,\dots,\cC_t \subset \Lambda \backslash S$ and a sequence $\mathcal{S}_0,\ldots, \mathcal{S}_t$ such that for all $j$, $\mathcal S_j$ is a collection of subsets of $\Lambda\setminus S$ satisfying
\begin{equation}\label{eq:sizeofseeds}
    \sum_{B\in \mathcal{S}_j}|B|\leq b\ell.
\end{equation}
Start with
$$\mathcal{S}_0=\varnothing \quad \text{and} \quad \mathcal{C}_0=\partialin \Lambda.$$
Now let $i \in \{1,\dots,t\}$, and suppose that we have defined $\mathcal S_j$ and $\mathcal C_j$ for all $j < i$, but we have not yet defined $\cS_i$ or $\cC_i$. Consider
\[
\mathcal D := \bigcup_{j=0}^{i-1} \bigcup_{B \in \mathcal S_j} B \quad \text{and} \quad \mathcal K := \bigcup_{j=0}^{i-1} \mathcal C_j.
\]
If $|\partial S\cap \partial \mathcal K| \geq t$, then set
\[
    \mathcal S_j := \varnothing \quad \text{and} \quad \mathcal C_j := \varnothing \quad \text{for all $j \geq i$.}
\]
Now instead suppose that $|\partial S\cap \partial \mathcal K|<t$. By applying~\eqref{eq:sizeofseeds} to all $j<i$, we have that
\[
    |\cD|\le b\ell i \leq b\ell t.
\]
Thus, $\cD$ and $\cK$ satisfy~\eqref{eq:sat}. So $\mathsf{Seeds}(\mathcal D,\mathcal K)$ is a sequence balls in $G[\Lambda \backslash S]$ with the properties guaranteed by Lemma~\ref{lem:seeds}, and the cluster $\mathcal C_{\cD , \cK}(B)$ from Section~\ref{s: growing} is well-defined for any $B \in \mathsf{Seeds}(\cD,\cK)$. Define
\[
    \mathcal S_{i} := \mathsf{Seeds}(\mathcal D,\mathcal K),
\]
and note that this choice satisfies~\eqref{eq:sizeofseeds} with $j=i$.
Let $\mathcal B$ be the first ball in the (ordered) sequence of balls given by $\mathsf{Seeds}(\cD,\cK)$ such that $\zeta(e) =1$ for all $e \in \mathcal B^{\circ}$; if no such ball exists, then set $\mathcal B := \varnothing$. If $\mathcal C_{\cD,\cK}(\mathcal B)$ intersects $\partialout S \setminus \cK$ and $\cK$, then set
\[
    C_{i} := \mathcal C_{\cD,\cK}(\mathcal B);
\]
otherwise, set
\[
    \mathcal C_{j} := \varnothing \quad \text{for all $j \geq i$}\qquad\text{ and }\qquad\mathcal S_{j} := \varnothing  \quad \text{for all $j \geq i+1$}.
\]

\noindent\textbf{Properties of $\cC_0,\dots,\cC_t$}
\\Note that for all $i$, if $\mathcal C_i = \varnothing$ then $\mathcal C_j = \varnothing$ for all $j > i$. Moreover, for all $i \in \{1,\dots,t\}$, if $\mathcal C_j \not= \varnothing $ for all $j \leq i$, then by Remark~\ref{rem:Cconnected} and our choice of $\mathcal C_0$, we have $\mathcal C_i \subset \mathcal C_{\partial \Lambda}$. Therefore,
\[
    \mathcal C_0,\dots,\mathcal C_t \subset \mathcal C_{\partial \Lambda}.
\]
Let $i \in \{1,\dots,t\}$. By construction of $\mathcal C_i$, we have that if $\mathcal C_i \not= \varnothing$ then $\mathcal C_i$ intersects $\partialout S \backslash \bigcup_{j<i} \mathcal C_j$. In particular,
\[
    N_i > N_{i-1} \quad \text{if and only if} \quad \cC_i \neq \varnothing,
\]
where $N_i$ is as given in~\eqref{eq:def_of_Ni}.
So in order to establish \eqref{eq:touches_++}, it suffices to show
\begin{equation}\label{eq: checkmate}
    \mathbb{P}( \cC_i \neq \varnothing \mid  \mathcal C_j \not= \varnothing \text{ for all } j \leq  i-1 \text{ and }N_{i-1} < t ) \geq 1-\eps,
\end{equation}
whenever the event being conditioned on has non-zero probability. To this end, fix non-empty sets of vertices $C_0,\dots,C_{i-1}$ such that $\mathbb{P}(\cC_j = C_j \text{ for all } j \leq i-1) >0$,
and such that $|\partial K \cap \partial S| < t$ where $K:= \bigcup_{j=0}^{i-1} C_j$. (If no such sets exist, then there is nothing to prove.) By averaging over all choices for these sets, to establish (\ref{eq: checkmate}), it suffices to show that
\begin{equation} \label{eq:eat_king}
    \mathbb{P}(\cC_i \not= \varnothing \mid \cC_j = C_j \text{ for all } j \leq i-1) \geq 1-\eps.
\end{equation}

Notice that there exists a unique sequence of deterministic collections $S_0,\dots,S_{i}$ such that if $\mathcal C_j = C_j$ for all $j \leq i-1$,  then $\cS_j = S_j$ for all $j \leq i$.
Now define
\[
D\coloneqq \bigcup_{j=0}^{i-1} \bigcup_{B \in S_j}B\quad \text{and}\quad K\coloneqq \bigcup_{j=0}^{i-1}C_{j}.
\]
Note that by repeatedly applying Remark~\ref{rem:Cmeas}, the event that $\cC_j = C_j$ for all $j \leq i-1$ is measurable with respect to
\begin{equation}\label{eq:triplemeas}
   \zeta \cap D^\circ, \quad \omega \cap \overline K, \quad \xi \cap K^\circ.
\end{equation}
Note that these three configurations are independent of
\begin{equation} \label{eq:zeta_on_new_balls}
\zeta \cap \bigcup_{B\in S_i} S_i^{\circ}
\end{equation}
because $\mathsf{Seeds}(D,K)$ only outputs balls that are disjoint from $D$. 
So, as in the construction of $\cC_i$, letting $\cB$ be the first $\zeta$-open ball in $\cS_i$ ($\cB=\varnothing$ if none exists), by our choice of $\ell$ in~\eqref{eq: eta},
\begin{equation}\label{eq:probagoodball}
    \mathbb P( \mathcal B = \varnothing \mid \cC_j = C_j \text{ for all } j \leq i-1) = \prod_{B \in S_i} (1 - \eta^{|B^{\circ}|}) \leq (1 - \eta^{b})^\ell  \leq \frac{\eps}{2}.
\end{equation}
Now consider a fixed $B \in \mathcal S_i$. By Remark~\ref{lem:equivalent1}
\begin{align}
&\mathbb{P}(\cC_i \not=\varnothing \mid  \cC_j = C_j \text{ for all } j \leq i-1\text{ and } \mathcal B = B)  
\\&= \mathbb{P}(    \partialout S\setminus K \xleftrightarrow[]{\omega\setminus \overline{K}} B \xleftrightarrow[]{(\omega\setminus\overline{K})\cup (\xi\cap \partial K)} K \mid  \cC_j = C_j \text{ for all } j \leq i-1 \text{ and }\mathcal B = B).
\end{align}
Note that the event that \[\partialout S\setminus K \xleftrightarrow[]{\omega\setminus \overline{K}} B \xleftrightarrow[]{(\omega\setminus\overline{K})\cup (\xi\cap \partial K)} K\] is measurable with respect to \[\omega\setminus \overline{K} \quad \text{and} \quad \xi \cap \partial K,\]
whereas the event that $\cC_j = C_j$ for all $j \leq i-1$ and $\mathcal B = B$ is measurable with respect to the configurations in~\eqref{eq:triplemeas} and \eqref{eq:zeta_on_new_balls}. So these two events are independent. Hence,
\begin{equation}\label{eq:rewritingconditional}
\mathbb{P}(\cC_i \not=\varnothing \mid  \cC_j = C_j \text{ for all } j \leq i-1 \text{ and } \mathcal B = B ) =\mathbb{P}(    \partialout S\setminus K \xleftrightarrow[]{\omega\setminus \overline{K}} B \xleftrightarrow[]{(\omega\setminus\overline{K})\cup (\xi\cap \partial K)} K ).
\end{equation}
So by Lemma~\ref{lem:Bconnectstoblahwithgoodproba} and averaging over all choices of $B$,
\begin{equation}
    \mathbb P( \mathcal C_i \not= \varnothing \mid  \cC_j = C_j \text{ for all } j \leq i-1 \text{ and }\mathcal B \not=\varnothing) \geq 1-\frac{\eps}{2}.\label{eq: bop bop}
\end{equation}
By combining equations~\eqref{eq:probagoodball}, \eqref{eq: bop bop}, and using that $(1-\eps/2)^2 \geq 1-\eps$, we establish~\eqref{eq:eat_king}, as desired. \qed

\section{Deducing Theorem \ref{thm: exp phi}}\label{s: deriving theorems}
In this section, we deduce Theorem \ref{thm: exp phi} from Theorem \ref{thm: main}. The first inequality appearing in Theorem \ref{thm: exp phi}, concerning the volume of $\cC_0$, is essentially an immediate corollary of Theorem \ref{thm: main}. On the other hand, to deduce the second inequality from Theorem \ref{thm: main}, concerning the radius of $\cC_0$, we will require the following geometric proposition.
 

\begin{proposition}\label{prop: geom}
Let $G$ be an infinite bounded-degree graph. Suppose that there exists $\eps > 0$ such that $\Phi(n)\ge \eps n^{1/2}$ for all $n\ge 1$. Then there exists $\delta>0$ such that for every finite, connected set of vertices $A$,
\begin{align}
    |\partial A|\ge \delta\cdot \operatorname{diam}(A).
\end{align}
\end{proposition}
This proposition is sharp in the sense that the hypothesis $\Phi(n) \geq \eps n^{1/2}$ cannot be weakened. Indeed, consider the graph induced by $\{(x,y)\colon |y|\le f(x)\}$ in $\mathbb{Z}^2$ for functions $f$ that are only \textit{slightly} sublinear. Before proving the proposition, let us conclude our proof of Theorem \ref{thm: exp phi}.

\begin{proof}[Proof Theorem \ref{thm: exp phi} given Proposition \ref{prop: geom}]
The first displayed inequality in Theorem \ref{thm: exp phi} follows  from Theorem \ref{thm: main} applied with $S=\{o\}$, because for all $n\ge 1$, if $n\le |\mathcal{C}_o|<\infty$ then $\Phi(\mathcal{C}_o)\ge \Phi(n)$.

As for the second displayed inequality in Theorem \ref{thm: exp phi}, let us start by explaining why we may assume without loss of generality that there exists a constant $\eps > 0$ such that $\Phi(n)\ge \eps n^{1/2}$ for all $n$: by \cite[Lemma 7.2]{LMS08} (see also \cite[Chapter 6.7]{Prob_on_trees_and_networks}) every infinite transitive graph that does not satisfy such a two-dimensional isoperimetric inequality cannot have at least two-dimensional volume growth, and hence by Gromov and Trofimov \cite{G81,T03} (see also \cite[Theorem 5.11]{W00}) is quasi-isometric to $\mathbb{Z}$, and therefore satisfies $p_c=1$. Now the second displayed inequality in Theorem~\ref{thm: exp phi} follows from Theorem~\ref{thm: main} because for all $n$, if $o \leftrightarrow \partial \mathrm{B}_n$ but $o \centernot\leftrightarrow \infty$, then $\operatorname{diam}(\cC_o) \geq n$, and hence by Proposition~\ref{prop: geom}, $\Phi(C_o) \geq \delta n$ for some constant $\delta >0$.
\end{proof}

All that remains is to prove the geometric proposition itself.

\begin{proof}[Proof of Proposition \ref{prop: geom}]
Let $A$ be a finite, connected set of vertices, and let $m$ be its diameter. Fix some vertex $o\in A$. Given integers $k$ and $r$, let $A_{k,r}$ be the set of all vertices $v \in A$ such that $\operatorname{dist}(o,v) \in [k-r,k+r]$, and let $A_k\coloneqq A_{k,0}$. Define
\[
    f(k,r) := |A_{k,r}| \quad \text{and} \quad g(k,r) := | \partial A_{k,r} \cap \partial A|.
\]

Given $k \in [m] := \{0,\dots,m\}$, let $r(k)$ be the smallest non-negative integer $r$ such that
\begin{equation} \label{eq:g_is_big_relative_to_f}
    g(k,r) \geq \frac{\eps}{2} f(k,r)^{1/2},
\end{equation}
which exists because by applying our isoperimetric inequality to $A$,
\[
    g(k,m) = |\partial A| \geq \eps |A|^{1/2} = \eps f(k,m)^{1/2}.
\]
Without loss of generality, assume that $\eps < 1$.
Let us prove by induction that for all $r \in [r(k)]$,
\begin{equation} \label{eq:quadratic_growth}
    f(k,r) \geq \left( \frac{\eps}{6d} \right)^2(r+1)^2,
\end{equation}
where $d$ is the maximum vertex degree of $G$. When $r = 0$, (\ref{eq:quadratic_growth}) holds trivially because $A$ is connected and hence $f(k,0) \geq 1$. Now let $r \in \{1,\dots,m\}$, and suppose that (\ref{eq:quadratic_growth}) holds at $r-1$. By applying our isoperimetric inequality to $A_{k,r-1}$,
\[
    d |A_{k-r}| + d |A_{k+r}| + g(k,r-1) \geq |\partial A_{k,r-1}| \geq \eps f(k,r-1)^{1/2}.
\]
So by the negation of (\ref{eq:g_is_big_relative_to_f}) and by (\ref{eq:quadratic_growth}) at $r-1$,
\[
    |A_{k-r}| + |A_{k+r}| \geq \frac{\eps}{2d} f(k,r-1)^{1/2} \geq \frac{\eps}{2d} \cdot \frac{\eps}{6d}r \geq \left( \frac{\eps}{6d} \right)^2(2r+1),
\]
and hence,
\begin{align}
    f(k,r) &= f(k,r-1) + |A_{k-r}|+|A_{k+r}|\\
    &\geq \left( \frac{\eps}{6d} \right)^2 r^2 + \left( \frac{\eps}{6d} \right)^2(2r+1) = \left( \frac{\eps}{6d} \right)^2(r+1)^2,
\end{align}
completing our induction. Now by (\ref{eq:g_is_big_relative_to_f}) and (\ref{eq:quadratic_growth}) at $r=r(k)$,
\begin{equation} \label{eq:g_is_linear}
    g(k,r(k)) \geq \frac{\eps^2}{12d} (r(k) + 1).
\end{equation}

Note that the collection of intervals
\[
    I_k := \left[k-r(k)-1/2, k + r(k) + 1/2\right] \quad \text{for $k \in [m]$}
\]
trivially covers $[0,m]$. Consider a minimal family of such intervals covering $[0,m]$. Enumerate the intervals in this family, from left to right, according to each interval's left endpoint, and observe that by minimality, the even-indexed intervals are disjoint and the odd-indexed intervals are disjoint. Hence there exists a subset $L \subset [m]$ such that the intervals $I_k$ for $k \in L$ are disjoint from each other and satisfy $\sum_{k \in L} |I_k| \ge m/2$. (We could have alternatively used Vitali's covering lemma.) So by (\ref{eq:g_is_linear}) and the fact that each $I_k$ has length $|I_k| \leq 2(r(k) + 1)$, 
\begin{align}
    |\partial A| \geq \sum_{k \in L} g(k,r(k)) \geq \sum_{k \in L} \frac{\eps^2}{12d}(r(k)+1) \geq \frac{\eps^2}{24d} \sum_{k \in L} |I_k| \geq \frac{\eps^2}{48d}m. &\qedhere
\end{align}
\end{proof}

\section{Collecting mass}\label{s: collecting mass}

Let $G$ be an infinite connected transitive graph. By Theorem \ref{thm: main}, for all $p > p_c$, there exists a constant $c>0$ such that for every finite set of vertices $S$,
\[
    \mathrm P_p(S \leftrightarrow \infty) \geq 1-e^{-c\,\Phi(S)},
\]
which is sharp up to the choice of $c$. The following theorem says that with high probability, $S$ is also connected to many vertices within its neighbourhood at every scale. Formally, given non-empty finite set of vertices $S$ and $n \in \mathbb N$, define
\begin{equation}
    \mathcal{C}_S^{n}=\big\{v\in \mathrm{B}_n(S): v\xleftrightarrow[]{\mathrm{B}_n(S)} S\big\}.
\end{equation}

\begin{theorem}\label{prop:collectingmass}
For all $p > p_c$ there exist $c>0$ such that for every finite non-empty set of vertices $S$ we have,
\begin{equation}\label{eq:collectingoal}
    \mathrm P_p\big(\forall n \in \mathbb N \quad |\mathcal{C}_S^n| \geq v_n\big) \geq 1-\sum_{n\ge |S|}e^{-c\Phi(n)},
\end{equation}
where  $v_n=v_n(S)$ for $n\ge 0$ is defined implicitly by  
\begin{equation}\label{eq:1}
    \int_{|S|}^{v_n}\frac 1{c\, \Phi(t)}\,dt= n.
\end{equation}
 \end{theorem}

\begin{remark}
Theorem~\ref{prop:collectingmass} may be viewed as an analogue for percolation of the well-known relationship between return probabilities and isoperimetry for random walks (see \cite[Chapter 6.7]{Prob_on_trees_and_networks}). 
\end{remark}
\begin{remark}
    Essentially, $v_n(S)$ is the lower bound on the size of the neighbourhood $\mathrm{B}_n(S)$ that one obtains simply by recursively applying the isoperimetric function to smaller neighbourhoods. If $G$ is of polynomial growth then $v_n(S)\asymp \mathrm{B}_n(S)$. On the other hand, if $G$ is of exponential growth, then by a result of Coulhon and Saloff-Coste, at worst $\Phi(x)\asymp x/\log x$, in which case $v_n\asymp e^{c\sqrt{n}}$. If $G$ is non-amenable (i.e. $\Phi(x)\gtrsim x$), then we always have $v_n\gtrsim e^{cn}$. 
\end{remark}

\begin{proof}
Let $p \in (p_c,1)$ and let $c'>0$ be the constant guaranteed by Theorem~\ref{thm: exp phi} (it suffices to consider only $p<1$). We will show that the conclusion of our theorem holds with $c :=\min\big(\frac{-c'}{4d\log(1-p)},\frac{1}{2d}\big)$, where $d$ is the vertex degree of $G$. Let $S$ be a finite set of vertices, and let $v_n$ be defined by~\eqref{eq:1}. Consider some non-negative integer $n$.
Since $\Phi$ is increasing, 
\begin{equation}
    v_{n+1} - v_n \le \int_{v_n}^{v_{n+1}} \frac{\Phi(v_{n+1})}{\Phi(t)} \, dt = c\, \Phi(v_{n+1}).
\end{equation}
By rearranging this inequality and using that $c \le 1/2d$ and $\Phi(v_{n+1}) \le d v_{n+1}$, we obtain $v_{n+1} \le 2v_n$. Since $
\Phi(2x) 
\leq 2 \Phi(x)$ for all $x$, it follows that $\Phi(v_{n+1}) \le \Phi(2v_n) \le 2 \Phi(v_n)$. Therefore, 
\begin{equation}\label{eq:vnestimate}
v_{n+1} - v_n \le 2 c\, \Phi(v_n).
\end{equation}

For each $n\ge 0$, let $M_n=|\mathcal C_S^n|$. Observe that $M_0=|S|=v_0$. By a union bound,
\begin{equation}\label{eq: unionbound}
    \mathrm{P}_p(\forall n \;\; M_n\ge v_n) 
    \geq 1 - \sum_{n\geq 0} \mathrm{P}_p(M_n \geq v_n \text{ but } M_{n+1} < v_{n+1}).
\end{equation}
Since $M_n$ is always an integer, the summands here are zero whenever $[v_n,v_{n+1})$ does not contain an integer. So let $I$ be the set of all $n\in \mathbb{N}$ such that the interval $[v_n, v_{n+1})$ contains an integer, and consider a fixed $n\in I$. Let $\mathcal W$ be the set of all edges $e \in \partial \mathrm{B}_{n+1}(S)$ such that $e \xleftrightarrow{\mathrm{B}_{n+1}(S)} S$. By~\eqref{eq:vnestimate}, when $M_n \geq v_n$ and $M_{n+1} < v_{n+1}$, we know that
\[
    M_{n+1} - M_n \leq 2c\, \Phi(v_{n}),
\]
and hence $|\mathcal{W}| \le 2cd\Phi(v_n)$. By revealing our percolation configuration in $\mathrm{B}_{n+1}(S)$, we can determine the set $\mathcal W$ without revealing the status of its edges. Moreover, if every edge in $W$ is closed then $S$ is not connected to infinity.  So by independence
\begin{align}
    \mathrm{P}_p\big(M_n \ge v_n \text{ but } M_{n+1} < v_{n+1}\big) \cdot (1-p)^{2cd\Phi(v_n)} \nonumber \le \mathrm{P}_p(v_n \le |\mathcal{C}_S| < \infty).
\end{align}
So by rearranging, since $c'$ has the properties guaranteed by Theorem \ref{thm: exp phi},
\begin{equation}\label{eq:rearrange}
    \mathrm{P}_p \big( M_n \ge v_n \text{ but } M_{n+1}< v_{n+1} \big) 
    \le \frac{e^{-c'\Phi(v_n)}}{(1-p)^{2cd \Phi(v_n)}} 
    \le e^{-c\Phi(v_n)},
\end{equation}
where the last inequality follows from our formula for $c$.
Finally, plugging~\eqref{eq:rearrange} in~\eqref{eq: unionbound} yields
\begin{align}
    \mathrm{P}_p(\forall n\;\; M_n\ge v_n) 
    &\ge 1 - \sum_{n\in I} e^{-c\,\Phi(v_n)} \ge 1 - \sum_{n \ge |S|} e^{-c\,\Phi(n)}.&\qedhere
\end{align}
\end{proof}

\begin{remark}
The argument in this section can be seen as a version of the `isoperimetry from tail estimates' argument in~\cite{Gabor}, restricted to balls.
\end{remark}

\section{Supercritical sharpness on $\mathbb Z^d$}
\label{sec:supercr-sharpn-mathb}

In this section, we will explain how to derive the following Grimmett--Marstrand slab theorem from Theorem~\ref{prop:collectingmass}. Fix $d \ge 3$ and consider percolation on the hypercubic lattice $\mathbb{Z}^d$. For each $\ell\ge 1$, consider a \emph{slab of thickness $\ell$}, given by
\begin{equation}
\mathbb S_\ell:=\mathbb Z^2\times\{-\ell,\ldots,\ell\}^{d-2}
\end{equation}

\begin{theorem}[Grimmett--Marstrand, 1990]\label{thm:1}
    For all $p > p_c(\mathbb{Z}^d)$, there exists $\ell \ge 1$ such that
    \begin{equation}
        \mathrm{P}_p(0 \xleftrightarrow{\mathbb{S}_\ell} \infty) > 0.
    \end{equation}
\end{theorem}

We will apply the following well-known lemma, which follows from a standard static renormalisation scheme. We provide its proof for the reader's convenience. We write $\mathrm{B}_k(n)$ as shorthand for $\mathrm{B}_k(e_n)$ where $e_n := (n,0,\dots,0) \in \mathbb Z^d$ and we set $\mathrm{B}_k := \mathrm{B}_k(0)$.

\begin{lemma}\label{lem:uniqimpliesperco}
  Let $p \in [0, 1]$. If there exists a constant $C\ge 2$ such that 
  \begin{equation}
    \label{eq:uniqimpliesperco}
    \lim_{k \to \infty} \limsup_{n \to \infty} \mathrm{P}_p(\mathrm{B}_k \xleftrightarrow{\mathrm{B}_{Cn}} \mathrm{B}_k(n)) = 1,
  \end{equation}
  then there exists $\ell \ge 1$ such that $\mathrm{P}_p(0 \xleftrightarrow{\mathbb{S}_\ell} \infty) > 0$.
\end{lemma}

\begin{proof}
Suppose that we have a constant $C\ge 2$ satisfying \eqref{eq:uniqimpliesperco}.
By \cite{LSS97} (or a direct Peierls argument), there exists $\varepsilon > 0$ such that any $3C$-dependent percolation on the usual square lattice $\mathbb{Z}^2$ with marginals of at least $1-3\varepsilon$ has an infinite cluster almost surely. Then there exists $k < \infty$ such that
  \begin{equation}
    \label{eq:15}
    \limsup_{n \to \infty} \mathrm{P}_p(\mathrm{B}_k \xleftrightarrow{\mathrm{B}_{Cn}} \mathrm{B}_k(n)) > 1 - \varepsilon.
  \end{equation}
    By the uniqueness of the infinite cluster, there exists $n_0<\infty$ such that
  \begin{equation}
    \label{eq:9}
    \mathrm{P}_p(U) > 1 - \varepsilon,
  \end{equation}
  where $U$ is the event that there exists at most one cluster in the restriction of our percolation configuration to $\mathrm{B}_{n_0}^\circ$ that intersects both $\mathrm{B}_k$ and $\partialin \mathrm{B}_{n_0}$. Hence by equation~\eqref{eq:15}, there exists $n \ge n_0$ such that
  \begin{equation}
    \label{eq:17}
    \mathrm{P}_p(\mathrm{B}_k \xleftrightarrow{\mathrm{B}_{Cn}} \mathrm{B}_k(n)) > 1 - \varepsilon.
  \end{equation}

    For each edge $uv$ of $\mathbb Z^2 \subset \mathbb Z^d$ (using the obvious inclusion), let $\zeta(uv)$ denote the indicator of the event
    \[
        \{(nu + \mathrm{B}_k) \xleftrightarrow{nu+\mathrm{B}_{Cn}} (nv + \mathrm{B}_k)\} \cap U(nu) \cap U(nv),
    \]
    where $U(x)$ denotes the event $U$ translated to be centred around the vector $(x,0,\dots,0)$. Note that $\zeta$ is a $3C$-dependent percolation with marginals at least $1-3\eps$. Moreover, a.s., if $\zeta$ contains an infinite cluster then the restriction of our original percolation configuration to $\mathbb S_{Cn}$ does too. So by our choice of $C$, we have $\mathrm{P}_p(0 \xleftrightarrow{\mathbb{S}_{Cn}} \infty) > 0$.
\end{proof}

\begin{proof}[Proof of Theorem~\ref{thm:1}] 
Let $p > p_c(\mathbb{Z}^d)$.
Since we are working on $\mathbb Z^d$, we know that $\Phi(n) \gtrsim n^{\frac{d-1}{d}}$. 
So by Theorem~\ref{prop:collectingmass} and some calculus, there exists $\delta > 0$ such that 
\begin{equation}
    \label{eq:7}
    \lim_{k \to \infty} \inf_{n \ge 3k} \mathrm{P}_p(|\mathcal{C}_{\mathrm{B}_k}^n| >\delta |\mathrm{B}_{2n}|) = 1.
\end{equation}
Let $C\coloneqq 2\lceil 1/\delta\rceil$. We will show that~\eqref{eq:uniqimpliesperco} holds with $2C$, which is sufficient to conclude our proof thanks to Lemma~\ref{lem:uniqimpliesperco}. To this end, let $\varepsilon\in (0,1)$. Pick $k \ge 1$ such that for all $n \geq 3k$,
\begin{equation}
    \label{eq:8}
    \mathrm{P}_p(|\mathcal{C}_{\mathrm{B}_k}^n| > \delta |\mathrm{B}_{2n}|) > 1 - \eta,
\end{equation}
where $\eta := \varepsilon^{C^2}/C$. Now let $n \geq 3k$. For each $i \in \{0, \ldots, C-1\}$, define
\[
    \mathcal{C}_i := \{x \in \mathrm{B}_{2Cn} : x \xleftrightarrow{\mathrm{B}_{Cn}(ni)} \mathrm{B}_k(ni)\}.
\]
Since $\mathrm{B}_{Cn}(ni) \subset \mathrm{B}_{2Cn}$, by~\eqref{eq:8} applied with $Cn$ in place of $n$ and a union bound,
\[
    \mathrm P_p\big( |\cC_i| > \delta |\mathrm{B}_{2Cn}|\text{ for all $i$}\big) \geq 1-C\eta.
\]
If $|\cC_i|\ge \delta |\mathrm{B}_{2Cn}|$ for all $i$, by a density argument, there exist indices $i \neq j$ such that $\mathcal{C}_i$ intersects $\mathcal{C}_j$. Since $|ni-nj|\leq Cn$, we have $\mathrm B_{Cn}(nj)\subset \mathrm B_{2Cn}(nj)$ and therefore $\mathrm{B}_k(ni) \xleftrightarrow{\mathrm{B}_{2Cn}(ni)} \mathrm{B}_k(nj)$. Since there are only at most $C^2$ pairs of indices, by the $\sqrt{\mathsf{trick}}$, we deduce that there exists fixed indices $i\not= j$ such that
\begin{equation}
  \label{eq:13}
  \mathrm{P}_p(\mathrm{B}_k(ni) \xleftrightarrow{\mathrm{B}_{2Cn}(ni)} \mathrm{B}_k(nj)) > 1 - (C\eta)^{1/C^2}\geq 1-\varepsilon.
\end{equation}
So by translation invariance, $m\coloneqq|i-j|n \geq n$ satisfies
\[
    \mathrm P_p(\mathrm{B}_k \xleftrightarrow{\mathrm{B}_{2Cm}} \mathrm{B}_k(m)) \geq 1- \eps.
\]
Since $n$ was arbitrary, this establishes that~\eqref{eq:uniqimpliesperco} holds.
\qedhere
\end{proof}

\begin{remark}
Theorem \ref{thm:1} says that for all $p > p_c(\mathbb Z^d)$, we have percolation in a sufficiently thick slab. In particular, we certainly have percolation in the half-space $\mathbb H:=\mathbb N\times\mathbb {Z}^{d-1}$. Let us sketch how to deduce that we have percolation in the half-space directly from Theorem \ref{thm: main}, without going via slabs: Consider the region between $\mathrm B_n$ and $S := \partialin \mathrm B_{2n}$ for some large $n$. By Theorem \ref{thm: main}, with probability at least $1/2$, some constant proportion $c >0 $ of the vertices in $S$ are connected to $\mathrm{B}_{n}$. (Otherwise, by closing $\partial B_n$, we can force $\mathrm{B}_n \centernot\leftrightarrow \infty$ with probability at least $\frac{1}{2} \cdot(1-p)^{d c |S|}$, contradicting Theorem \ref{thm: main} if $c$ is small.) By considering the point of view of a uniformly random vertex in $S$ on this event, we see that $\mathrm P_p(0 \xleftrightarrow{\mathbb H} \infty) \geq c/2> 0$, as required. 
    
    
\end{remark}

\paragraph{Acknowledgements} We thank S\'{e}bastien Martineau for helpful comments on an earlier draft. BS was supported by SNSF grant 200021--228014. This project has received funding from the European Research Council (ERC) under the European Union’s Horizon 2020 research and innovation programme (grant agreement No 851565).

\bibliographystyle{alphaabbr}
\bibliography{perco}

@article{AB87,
 author = {Aizenman, Michael and Barsky, David J.},
 title = {Sharpness of the phase transition in percolation models},
 fjournal = {Communications in Mathematical Physics},
 journal = {Commun. Math. Phys.},
 issn = {0010-3616},
 volume = {108},
 pages = {489--526},
 year = {1987},
 language = {English},
 doi = {10.1007/BF01212322},
 zbMATH = {4003223},
 Zbl = {0618.60098}
}

@article{M86,
 author = {Men'shikov, M. V.},
 title = {Coincidence of critical points in percolation problems},
 fjournal = {Soviet Mathematics. Doklady},
 journal = {Sov. Math., Dokl.},
 issn = {0197-6788},
 volume = {33},
 pages = {856--859},
 year = {1986},
 language = {English},
 zbMATH = {3996823},
 Zbl = {0615.60096}
}

@article{AV08,
 author = {Antunovi{\'c}, Ton{\'c}i and Veseli{\'c}, Ivan},
 title = {Sharpness of the phase transition and exponential decay of the subcritical cluster size for percolation on quasi-transitive graphs},
 fjournal = {Journal of Statistical Physics},
 journal = {J. Stat. Phys.},
 issn = {0022-4715},
 volume = {130},
 number = {5},
 pages = {983--1009},
 year = {2008},
 language = {English},
 doi = {10.1007/s10955-007-9459-x},
 zbMATH = {5264682},
 Zbl = {1214.82028}
}

@article{DT16,
 author = {Duminil-Copin, Hugo and Tassion, Vincent},
 title = {A new proof of the sharpness of the phase transition for {Bernoulli} percolation and the {Ising} model},
 fjournal = {Communications in Mathematical Physics},
 journal = {Commun. Math. Phys.},
 issn = {0010-3616},
 volume = {343},
 number = {2},
 pages = {725--745},
 year = {2016},
 language = {English},
 doi = {10.1007/s00220-015-2480-z},
 zbMATH = {6579800},
 Zbl = {1342.82026}
}

@article {deducing_exp_from_GM,
    AUTHOR = {Chayes, J. T. and Chayes, L. and Newman, C. M.},
     TITLE = {Bernoulli percolation above threshold: an invasion percolation
              analysis},
   JOURNAL = {Ann. Probab.},
  FJOURNAL = {The Annals of Probability},
    VOLUME = {15},
      YEAR = {1987},
    NUMBER = {4},
     PAGES = {1272--1287},
      ISSN = {0091-1798,2168-894X},
   MRCLASS = {60K35 (82A43)},
  MRNUMBER = {905331},
MRREVIEWER = {H.\ Kesten},
       URL =
              {http://links.jstor.org/sici?sici=0091-1798(198710)15:4<1272:BPATAI>2.0.CO;2-2&origin=MSN},
}

@article{perco_beyond_zd,
 author = {Benjamini, Itai and Schramm, Oded},
 title = {Percolation beyond {{\(\mathbb{Z}^ d\)}}, many questions and a few answers},
 fjournal = {Electronic Communications in Probability},
 journal = {Electron. Commun. Probab.},
 issn = {1083-589X},
 volume = {1},
 pages = {71--82},
 year = {1996},
 language = {English},
 doi = {10.1214/ECP.v1-978},
 url = {https://eudml.org/doc/119473},
 zbMATH = {1122743},
 Zbl = {0890.60091}
}

@book {Prob_on_trees_and_networks,
    AUTHOR = {Lyons, Russell and Peres, Yuval},
     TITLE = {Probability on trees and networks},
    SERIES = {Cambridge Series in Statistical and Probabilistic Mathematics},
    VOLUME = {42},
 PUBLISHER = {Cambridge University Press, New York},
      YEAR = {2016},
     PAGES = {xv+699},
      ISBN = {978-1-107-16015-6},
   MRCLASS = {60C05 (05C05 05C81 28A80 60J10 60J80 60K35 82B41)},
  MRNUMBER = {3616205},
MRREVIEWER = {Laurent\ Miclo},
       DOI = {10.1017/9781316672815},
       URL = {https://doi.org/10.1017/9781316672815},
}

@article{FR08,
 author = {Fountoulakis, N. and Reed, B. A.},
 title = {The evolution of the mixing rate of a simple random walk on the giant component of a random graph},
 fjournal = {Random Structures \& Algorithms},
 journal = {Random Struct. Algorithms},
 issn = {1042-9832},
 volume = {33},
 number = {1},
 pages = {68--86},
 year = {2008},
 language = {English},
 doi = {10.1002/rsa.20210},
 zbMATH = {5320778},
 Zbl = {1147.60316}
}

@article{BKW14,
 author = {Benjamini, Itai and Kozma, Gady and Wormald, Nicholas},
 title = {The mixing time of the giant component of a random graph},
 fjournal = {Random Structures \& Algorithms},
 journal = {Random Struct. Algorithms},
 issn = {1042-9832},
 volume = {45},
 number = {3},
 pages = {383--407},
 year = {2014},
 language = {English},
 doi = {10.1002/rsa.20539},
 zbMATH = {6370216},
 Zbl = {1373.05172}
}

@misc{ADLZ25,
 author = {Anastos, Michael and Diskin, Sahar and Lichev, Lyuben and Zhukovskii, Maksim},
 title = {Diameter and mixing time of the giant component in the percolated hypercube},
 year = {2025},
 howpublished = {Preprint, {arXiv}:2510.13348 [math.{PR}] (2025)},
 url = {https://arxiv.org/abs/2510.13348},
 arXiv = {arXiv:2510.13348}
}

@article{EKK23,
 author = {Erde, Joshua and Kang, Mihyun and Krivelevich, Michael},
 title = {Expansion in supercritical random subgraphs of the hypercube and its consequences},
 fjournal = {The Annals of Probability},
 journal = {Ann. Probab.},
 issn = {0091-1798},
 volume = {51},
 number = {1},
 pages = {127--156},
 year = {2023},
 language = {English},
 doi = {10.1214/22-AOP1592},
 zbMATH = {7628800},
 Zbl = {1506.05115}
}

@article{BM03,
 author = {Benjamini, Itai and Mossel, Elchanan},
 title = {On the mixing time of a simple random walk on the super critical percolation cluster},
 fjournal = {Probability Theory and Related Fields},
 journal = {Probab. Theory Relat. Fields},
 issn = {0178-8051},
 volume = {125},
 number = {3},
 pages = {408--420},
 year = {2003},
 language = {English},
 doi = {10.1007/s00440-002-0246-y},
 zbMATH = {1964647},
 Zbl = {1020.60037}
}

@book {Grimmetts_perco_book,
    AUTHOR = {Grimmett, Geoffrey},
     TITLE = {Percolation},
    SERIES = {Grundlehren der mathematischen Wissenschaften [Fundamental
              Principles of Mathematical Sciences]},
    VOLUME = {321},
   EDITION = {Second},
 PUBLISHER = {Springer-Verlag, Berlin},
      YEAR = {1999},
     PAGES = {xiv+444},
      ISBN = {3-540-64902-6},
   MRCLASS = {60K35 (60-02 82B43)},
  MRNUMBER = {1707339},
MRREVIEWER = {Neal\ Madras},
       DOI = {10.1007/978-3-662-03981-6},
       URL = {https://doi.org/10.1007/978-3-662-03981-6},
}

@article{GM90,
 author = {Grimmett, G. R. and Marstrand, J. M.},
 title = {The supercritical phase of percolation is well behaved},
 fjournal = {Proceedings of the Royal Society of London. Series A. Mathematical and Physical Sciences},
 journal = {Proc. R. Soc. Lond., Ser. A},
 issn = {0080-4630},
 volume = {430},
 number = {1879},
 pages = {439--457},
 year = {1990},
 language = {English},
 doi = {10.1098/rspa.1990.0100},
 zbMATH = {4169828},
 Zbl = {0711.60100}
}

@article{HH21,
 author = {Hermon, Jonathan and Hutchcroft, Tom},
 title = {Supercritical percolation on nonamenable graphs: isoperimetry, analyticity, and exponential decay of the cluster size distribution},
 fjournal = {Inventiones Mathematicae},
 journal = {Invent. Math.},
 issn = {0020-9910},
 volume = {224},
 number = {2},
 pages = {445--486},
 year = {2021},
 language = {English},
 doi = {10.1007/s00222-020-01011-3},
 zbMATH = {7355937},
 Zbl = {1516.60057}
}

@article{CMT24,
 author = {Contreras, Daniel and Martineau, S{\'e}bastien and Tassion, Vincent},
 title = {Supercritical percolation on graphs of polynomial growth},
 fjournal = {Duke Mathematical Journal},
 journal = {Duke Math. J.},
 issn = {0012-7094},
 volume = {173},
 number = {4},
 pages = {745--806},
 year = {2024},
 language = {English},
 doi = {10.1215/00127094-2023-0032},
 zbMATH = {7858201},
 Zbl = {1552.60262}
}

@article{EST25,
 author = {Easo, Philip and Severo, Franco and Tassion, Vincent},
 title = {Counting minimal cutsets and {{\(p_c<1\)}}},
 fjournal = {Forum of Mathematics, Pi},
 journal = {Forum Math. Pi},
 issn = {2050-5086},
 volume = {13},
 pages = {13},
 note = {Id/No e23},
 year = {2025},
 language = {English},
 doi = {10.1017/fmp.2025.10011},
 zbMATH = {8110757}
}

@article {V25,
    AUTHOR = {Vanneuville, Hugo},
     TITLE = {Exponential decay of the volume for {B}ernoulli percolation: a
              proof via stochastic comparison},
   JOURNAL = {Ann. H. Lebesgue},
  FJOURNAL = {Annales Henri Lebesgue},
    VOLUME = {8},
      YEAR = {2025},
     PAGES = {101--112},
      ISSN = {2644-9463},
   MRCLASS = {60K35},
  MRNUMBER = {4912669},
}

@incollection {BLS99,
    AUTHOR = {Benjamini, Itai and Lyons, Russell and Schramm, Oded},
     TITLE = {Percolation perturbations in potential theory and random
              walks},
 BOOKTITLE = {Random walks and discrete potential theory ({C}ortona, 1997)},
    SERIES = {Sympos. Math.},
    VOLUME = {XXXIX},
     PAGES = {56--84},
 PUBLISHER = {Cambridge Univ. Press, Cambridge},
      YEAR = {1999},
      ISBN = {0-521-77312-1},
   MRCLASS = {60K35 (31C20 60G50)},
  MRNUMBER = {1802426},
MRREVIEWER = {Yu\ Zhang},
}

@article {H22,
    AUTHOR = {Hutchcroft, Tom},
     TITLE = {Transience and anchored isoperimetric dimension of
              supercritical percolation clusters},
   JOURNAL = {Electron. J. Probab.},
  FJOURNAL = {Electronic Journal of Probability},
    VOLUME = {28},
      YEAR = {2023},
     PAGES = {Paper No. 14, 15},
      ISSN = {1083-6489},
   MRCLASS = {60K35 (60J99)},
  MRNUMBER = {4536689},
MRREVIEWER = {Akira\ Sakai},
       DOI = {10.1214/23-ejp905},
       URL = {https://doi.org/10.1214/23-ejp905},
}

@article {G81,
    AUTHOR = {Gromov, Mikhael},
     TITLE = {Groups of polynomial growth and expanding maps},
   JOURNAL = {Inst. Hautes \'Etudes Sci. Publ. Math.},
  FJOURNAL = {Institut des Hautes \'Etudes Scientifiques. Publications
              Math\'ematiques},
    NUMBER = {53},
      YEAR = {1981},
     PAGES = {53--73},
      ISSN = {0073-8301,1618-1913},
   MRCLASS = {53C20 (22E40 58F15)},
  MRNUMBER = {623534},
MRREVIEWER = {J.\ A.\ Wolf},
       URL = {http://www.numdam.org/item?id=PMIHES_1981__53__53_0},
}

@article {T03,
    AUTHOR = {Trofimov, Vladimir I.},
     TITLE = {Undirected and directed graphs with near polynomial growth},
   JOURNAL = {Discuss. Math. Graph Theory},
  FJOURNAL = {Discussiones Mathematicae. Graph Theory},
    VOLUME = {23},
      YEAR = {2003},
    NUMBER = {2},
     PAGES = {383--391},
      ISSN = {1234-3099,2083-5892},
   MRCLASS = {05C25 (20F65)},
  MRNUMBER = {2070163},
MRREVIEWER = {Tatiana\ Smirnova-Nagnibeda},
       DOI = {10.7151/dmgt.1208},
       URL = {https://doi.org/10.7151/dmgt.1208},
}

@article{DCRT19,
 author = {Duminil-Copin, Hugo and Raoufi, Aran and Tassion, Vincent},
 title = {Sharp phase transition for the random-cluster and {Potts} models via decision trees},
 fjournal = {Annals of Mathematics. Second Series},
 journal = {Ann. Math. (2)},
 issn = {0003-486X},
 volume = {189},
 number = {1},
 pages = {75--99},
 year = {2019},
 language = {English},
 doi = {10.4007/annals.2019.189.1.2},
 zbMATH = {7003145},
 Zbl = {1482.82009}
}

@article{LMS08,
 author = {Lyons, Russell and Morris, Benjamin J. and Schramm, Oded},
 title = {Ends in uniform spanning forests},
 fjournal = {Electronic Journal of Probability},
 journal = {Electron. J. Probab.},
 issn = {1083-6489},
 volume = {13},
 pages = {1702--1725},
 year = {2008},
 language = {English},
 doi = {10.1214/EJP.v13-566},
 url = {https://eudml.org/doc/230776},
 zbMATH = {5636553},
 Zbl = {1191.60016}
}

@book {W00,
    AUTHOR = {Woess, Wolfgang},
     TITLE = {Random walks on infinite graphs and groups},
    SERIES = {Cambridge Tracts in Mathematics},
    VOLUME = {138},
 PUBLISHER = {Cambridge University Press, Cambridge},
      YEAR = {2000},
     PAGES = {xii+334},
      ISBN = {0-521-55292-3},
   MRCLASS = {60B15 (60G50 60J10)},
  MRNUMBER = {1743100},
MRREVIEWER = {Donald\ I.\ Cartwright},
       DOI = {10.1017/CBO9780511470967},
       URL = {https://doi.org/10.1017/CBO9780511470967},
}

@article {MR91567,
    AUTHOR = {Broadbent, S. R. and Hammersley, J. M.},
     TITLE = {Percolation processes. {I}. {C}rystals and mazes},
   JOURNAL = {Proc. Cambridge Philos. Soc.},
  FJOURNAL = {Proceedings of the Cambridge Philosophical Society},
    VOLUME = {53},
      YEAR = {1957},
     PAGES = {629--641},
      ISSN = {0008-1981},
   MRCLASS = {60.0X},
  MRNUMBER = {91567},
MRREVIEWER = {G.\ Newell},
       DOI = {10.1017/s0305004100032680},
       URL = {https://doi.org/10.1017/s0305004100032680},
}

@article {Peres,
    AUTHOR = {Chen, Dayue and Peres, Yuval},
     TITLE = {Anchored expansion, percolation and speed},
      NOTE = {With an appendix by G\'abor Pete},
   JOURNAL = {Ann. Probab.},
  FJOURNAL = {The Annals of Probability},
    VOLUME = {32},
      YEAR = {2004},
    NUMBER = {4},
     PAGES = {2978--2995},
      ISSN = {0091-1798,2168-894X},
   MRCLASS = {60C05 (60D05 60G50 60K35 60K37)},
  MRNUMBER = {2094436},
MRREVIEWER = {Timo\ Sepp\"al\"ainen},
       DOI = {10.1214/009117904000000586},
       URL = {https://doi.org/10.1214/009117904000000586},
}

@article {Gabor,
    AUTHOR = {Pete, G\'abor},
     TITLE = {A note on percolation on {$\Bbb Z^d$}: isoperimetric profile
              via exponential cluster repulsion},
   JOURNAL = {Electron. Commun. Probab.},
  FJOURNAL = {Electronic Communications in Probability},
    VOLUME = {13},
      YEAR = {2008},
     PAGES = {377--392},
      ISSN = {1083-589X},
   MRCLASS = {60K35 (82B43 82D30)},
  MRNUMBER = {2415145},
       DOI = {10.1214/ECP.v13-1390},
       URL = {https://doi.org/10.1214/ECP.v13-1390},
}

@article {DisorderEntropy,
    AUTHOR = {Benjamini, Itai and Duminil-Copin, Hugo and Kozma, Gady and
              Yadin, Ariel},
     TITLE = {Disorder, entropy and harmonic functions},
   JOURNAL = {Ann. Probab.},
  FJOURNAL = {The Annals of Probability},
    VOLUME = {43},
      YEAR = {2015},
    NUMBER = {5},
     PAGES = {2332--2373},
      ISSN = {0091-1798,2168-894X},
   MRCLASS = {60K37 (31A05 37A35 60B15 60J10 82B43)},
  MRNUMBER = {3395463},
MRREVIEWER = {Daniel\ Boivin},
       DOI = {10.1214/14-AOP934},
       URL = {https://doi.org/10.1214/14-AOP934},
}

@book{BO06,
 author = {Bollob{\'a}s, B{\'e}la and Riordan, Oliver},
 title = {Percolation.},
 isbn = {0-521-87232-4},
 year = {2006},
 publisher = {Cambridge: Cambridge University Press},
 language = {English},
 doi = {10.1017/CBO9781139167383},
 zbMATH = {5076923},
 Zbl = {1118.60001}
}

@article {Gaboriau,
    AUTHOR = {Gaboriau, D.},
     TITLE = {Invariant percolation and harmonic {D}irichlet functions},
   JOURNAL = {Geom. Funct. Anal.},
  FJOURNAL = {Geometric and Functional Analysis},
    VOLUME = {15},
      YEAR = {2005},
    NUMBER = {5},
     PAGES = {1004--1051},
      ISSN = {1016-443X,1420-8970},
   MRCLASS = {60K35 (37A20)},
  MRNUMBER = {2221157},
MRREVIEWER = {Filippo\ Cesi},
       DOI = {10.1007/s00039-005-0539-2},
       URL = {https://doi.org/10.1007/s00039-005-0539-2},
}

@article {LSS97,
    AUTHOR = {Liggett, T. M. and Schonmann, R. H. and Stacey, A. M.},
     TITLE = {Domination by product measures},
   JOURNAL = {Ann. Probab.},
  FJOURNAL = {The Annals of Probability},
    VOLUME = {25},
      YEAR = {1997},
    NUMBER = {1},
     PAGES = {71--95},
      ISSN = {0091-1798,2168-894X},
   MRCLASS = {60G60 (60G10 60K35)},
  MRNUMBER = {1428500},
MRREVIEWER = {Vincent\ De Valk},
       DOI = {10.1214/aop/1024404279},
       URL = {https://doi.org/10.1214/aop/1024404279},
}

\end{document}